\documentclass[12pt,reqno,oneside]{amsart}
\usepackage{amsmath,amssymb,amsthm}
\usepackage{hyperref}
\hypersetup{
    colorlinks=true,
    linkcolor=blue,
    filecolor=magenta,      
    urlcolor=cyan,
    citecolor=blue,
    pdftitle={Toric Constructions via Integral Transversality},
    }
\usepackage{stmaryrd}
\usepackage{bbm}
\usepackage{array} 
\usepackage{graphicx}
\usepackage[english,capitalise]{cleveref} 
\usepackage{multirow}

\calclayout

\usepackage{tikz-cd}
\usepackage{tikz-3dplot}
\usetikzlibrary{arrows}
\tdplotsetmaincoords{70}{110}

\numberwithin{equation}{section}

\newtheorem{lemma}[equation]{Lemma}  
\newtheorem{corollary}[equation]{Corollary}
\newtheorem{theorem}[equation]{Theorem}
\newtheorem{proposition}[equation]{Proposition}

\theoremstyle{definition}
\newtheorem{remark}[equation]{Remark}
\newtheorem{remarks}[equation]{Remarks}
\newtheorem{definition}[equation]{Definition}
\newtheorem{example}[equation]{Example}

\newcolumntype{P}[1]{>{\centering\arraybackslash}p{#1}}

\newif\ifdebug                                                      %
\debugtrue

\def\AA{{\mathbb A}}
\def\CC{{\mathbb C}}
\def\PP{{\mathbb P}}
\def\RR{{\mathbb R}}
\def\ZZ{{\mathbb Z}}
\def\HH{{\mathbb H}}
\def\NN{{\mathbb N}}
\def\fg{{\mathfrak g}}
\def\fh{{\mathfrak h}}

\def\ft{{\mathfrak t}}
\def\fs{{\mathfrak s}}
\def\cV{{\mathcal V}}
\def\Ii{{\mathcal I}}
\newcommand{\fgz}{\fg_{_\ZZ}}
\newcommand{\ftz}{\ft_{_\ZZ}}
\newcommand{\fhz}{\fh_{_\ZZ}}

\def \red {\text{red}}
\def \cut {\text{cut}}

\def \ker {\text{ker}}

\def \skew {\text{skew}}
\def \Lie {\text{Lie}}
\def \sp {\text{Sp}}
\def \Hom {\operatorname{Hom}}
\def \opker {\operatorname{ker}}

\def\dd{\text{d}}
\newcommand{\expval}[1]{\langle #1 \rangle}
\newcommand{\abs}[1]{\left| #1 \right|}

\newcommand{\dims}[1]{\text{dim}\left( #1 \right)}
\newcommand{\qq}[1]{\qquad \text{#1} \qquad}
\newcommand{\co}[1]{\text{Cone}\left( #1 \right)}

\newcommand{\dir}[1]{\textup{dir}\left( #1\right)}
\newcommand{\spn}[2]{\textup{Span}_{#1}\left( #2\right)}
\renewcommand{\mapsfrom}{\mathrel{\reflectbox{\ensuremath{\longmapsto}}}}

\title{Toric Constructions via Integral Transversality}


\author{Ana Cannas da Silva}  

\address{
D-MATH,
ETH Zürich, 
Rämistrasse 101,
8092 Zürich,
Switzerland }
\email{ana.cannas@math.ethz.ch}

\author{Reto Kaufmann}  

\address{
D-MATH,
ETH Zürich, 
Rämistrasse 101,
8092 Zürich,
Switzerland }
\email{reto.kaufmann@math.ethz.ch}

\date{August 31, 2026}

\begin{document}

\begin{abstract}
In the Delzant world, symplectic toric manifolds correspond to unimodular polytopes, and unimodularity-preserving polytope constructions correspond to geometric constructions on the manifold side.
We treat, in a unified framework, constructions such as
taking a face, taking a slice, and polytope intersections, identifying each with its geometric counterpart, namely invariant submanifolds, symplectic reduction, and reduction of a product (a full list appears in \cref{table:transformations_summary}).
The key tool is \textit{integral transversality}, a lattice-refined transversality condition that provides a uniform criterion for preserving unimodularity, ensuring smoothness on the manifold side.
To keep every construction within a single ambient vector space -- the dual of the Lie algebra of a fixed torus -- we allow the torus actions involved to have nontrivial, but connected kernels.
\end{abstract}
\maketitle

\thispagestyle{empty}

\tableofcontents

\section*{Introduction}

By Delzant's work~\cite{delzant1988hamiltoniens}, $n$-dimensional
unimodular\footnote{A polytope is \textit{unimodular} when all its vertex cones are smooth; cf.\ \Cref{sec:polyhedral_cones}.}
polytopes in the dual of the Lie algebra of a torus $G$
correspond one-to-one, via the moment map, to
isomorphism classes of $2n$-dimensional symplectic toric $G$-manifolds.
Under this correspondence, constructions with polytopes that preserve unimodularity become geometric constructions on the manifold side, and applying such constructions allows one to build new toric symplectic manifolds out of old ones;
see, for instance, \cite{LermanSymplecticCuts,lerman_four_author,karshon_lerman_noncompact}
as well as the books \cite{audin2012topology,cannas2004lectures,guillemin1994combinatorial}.

This paper treats, within a single symplectic framework, the seven constructions listed in \Cref{table:transformations_summary}:
taking a face, taking a slice, intersection, cutting, corner chopping, face chopping, and merging.\footnote{\textit{Merging} is related to \textit{coarsening} in Algebraic Geometry, where multiple cones in a fan are merged together to form a new fan,
as well as to Gompf's \textit{symplectic sums}.}
The first two lower the dimension of the polytope and the corresponding manifold; the last four preserve it.  Among the latter, cutting and merging are mutually converse -- one splits a polytope along a hyperplane, while the other glues two polytopes back together along a shared facet. Corner and face chopping are special cases of the cutting mechanism, localized at a single vertex or along an entire face.

Every construction on this list has already been studied by other authors in some form or is probably familiar to the community, even if not published.
This paper can also serve as a reference for otherwise scattered or folklore constructions.

Constructions that raise the dimension of the polytope, starting with products, are left for a separate paper.
\begin{center}
\renewcommand{\arraystretch}{1.33}
\begin{table}[ht!]
\begin{tabular}{|P{2cm}|P{3.85cm}|P{3cm}|P{2.35cm}|P{0.5cm}|}
\hline
\textbf{Construct.} &
\textbf{Change in polytope} &
\textbf{Geometric meaning} &
\textbf{Smoothness} &
\textbf{\S} \\
\hline \hline
Taking a face &
Take a face &
Invariant \mbox{submanifold} &
Always &
\ref{sec:submanifolds} \\
\hline
Taking a slice &
Intersect nontrivially with an affine subspace &
Symplectic \mbox{reduction} &
Yes when integrally transverse &
\ref{sec:reduction_levels} \\
\hline 
Intersection &
Intersect two polytopes &
Reduction of product &
Yes when integrally transverse &
\ref{sec:intersection} \\
\hline \hline
Cutting &
Intersect with an affine halfspace &
Compactification of invariant open submanifold &
Yes when integrally transverse &
\ref{sec:cutting_levels} \\
\hline
Corner chopping &
Intersect near a vertex with affine halfspace whose boundary is normal to sum of normals at that vertex &
Equivariant blow-up at a point &
Always &
\ref{sec:choppingcorner} \\
\hline
Face \quad chopping &
Intersect near a face with affine halfspace whose boundary is normal to sum of normals at that face &
Equivariant blow-up at a submanifold &
Always &
\ref{sec:choppingface}
\\
\hline
Merging &
Merge two polytopes along common facet &
Equivariant blow-down inv. codim 2 subm. &
Yes when mergeable
(Def.\ \ref{def:mergeable}) &
\ref{sec:merging} \\
\hline
\end{tabular}
\vspace{.5cm}
\caption{Constructions for unimodular polytopes, their geometric meaning for symplectic toric manifolds, and the paper sections where they are addressed.}
\label{table:transformations_summary}
\end{table}
\end{center}

The main technical engine that allows us to treat all constructions uniformly is a notion we call \textit{integral transversality}: a lattice-refined strengthening of ordinary transversality. This is a generalization of the condition introduced by McDuff~\cite{mcduff2011displacing}. While ordinary transversality only sees the linear structure of a polytope's faces, integral transversality also takes into account the lattice within each face.
It is exactly this finer condition -- checkable, as we show, at finitely many vertices -- that ensures that a given construction preserves unimodularity, and hence smoothness on the manifold side.

Several of the polytope constructions above lower the dimension of the resulting polytope without lowering the dimension of the ambient vector space. In particular, a translate through the origin of a polytope obtained this way need not span all of $\fg^*$. On the manifold side, this forces us to allow the torus action to have a nontrivial kernel: rather than restricting attention to effective actions, as is customary, we develop the theory for symplectic toric manifolds whose kernel is a subtorus, explicitly tracked throughout. Absorbing the change of dimension in this kernel instead of the ambient vector space is what allows every construction in \Cref{table:transformations_summary} to be carried out inside one fixed copy of $\fg^*$.

\medskip

The principal contributions of this paper are:
\begin{itemize}
    \item a significant generalization of the notion of integral transversality (\cref{def:int_transv,def:integral_transv_polytopes}) with corresponding criteria and 
    \item a systematic perspective on constructions with symplectic toric manifolds, viewed as operations on unimodular polytopes (cf.\ \Cref{table:transformations_summary}).
\end{itemize}
The contents of the paper are as follows:
\begin{itemize}
\item the annihilator mechanism in the context of tori to convert subgroups of the weight lattice into closed subgroups of the torus and vice versa (\Cref{sec:tori_framework});
\item  a framework for symplectic toric representations with nontrivial kernels, establishing the relation between nontrivial kernels and the generalized unimodularity condition (\cref{sec:toric_reps});
\item a framework for symplectic toric manifolds with nontrivial action kernels, including the toric local model underlying it (\Cref{sec:stm_framework});
\item a thorough treatment of reduction in terms of polytopes -- reduction levels, polytope intersection, cutting, and blow-ups -- together with merging, the converse of cutting (\Cref{sec:reduction_1,sec:reduction_2});
\item a substantial appendix collecting necessary background on polyhedral cones and polytopes (\Cref{sec:polyhedral_cones_and_polytopes});
\item a concise appendix summarizing relevant statements with references about Pontryagin duality and the annihilator mechanism (\Cref{sec:pontryagin}).
\end{itemize}
A more detailed content description is given
at the beginning of each section.

\subsection*{Dedication}
It is a great pleasure to dedicate this paper to Rui Loja Fernandes,
who has been a major driving force behind excellent developments
related to symmetry, integrability, and geometry
over the past three decades.

\subsection*{Acknowledgement}
This work grew out of the MSc thesis written by R.K. under the supervision of A.C. at ETH Zurich.
We gratefully acknowledge stimulating conversations with Yael Karshon.

\section{Tori and Lattices}
\label{sec:tori_framework}

Before introducing any symplectic geometry, we lay down the purely algebraic groundwork on which everything else rests.
After recalling the integral and weight lattices of a torus (\cref{sec:lattices}) and reviewing its closed subgroups (\cref{sec:closed}), we introduce
the annihilator mechanism as the technical tool to translate geometric operations on closed subgroups into algebraic operations on lattices (\cref{sec:annihilator}).
This dictionary is essential for the control of the connectedness of the action kernel via the weight lattices in \cref{sec:stm_framework,sec:reduction_1}. 



\subsection{Integral Lattice and Weight Lattice}
\label{sec:lattices}

Let $T$ be a torus, i.e. an abelian compact connected Lie group, and let $\ft$ be its Lie algebra. The exponential map $\exp_T:\ft\to T$ is a homomorphism since $T$ is abelian and surjective since $T$ is connected.
As $\exp_T$ is a local diffeomorphism, its kernel
\[
\ftz := \ker(\exp_T)
\]
is a discrete subgroup, called the {\em integral lattice}.
The above data fits into a short exact sequence of abelian Lie groups:
\begin{equation*}
    \begin{tikzcd}
0 \arrow[r] & \ftz \arrow[r, hook] & \ft \arrow[r, "\exp_T", two heads] & T \arrow[r] & 0.
\end{tikzcd}
\end{equation*}

\begin{lemma}
Let $G$ and $H$ be two tori.
Then their Lie group homomorphisms $\Hom(G,H)$
correspond one-to-one to the $\ZZ$-module homomorphisms
$\Hom_{_\ZZ}(\fgz,\fhz)$ between their respective integral lattices as follows:
\begin{align}
\label{eq:hom_correspondence}
        \Hom(G,H) &\longrightarrow \Hom_{_\ZZ}(\fgz,\fhz) \\
        F &\longmapsto f_{_\ZZ} := \dd F_e\vert_{\fgz}. \nonumber
    \end{align}
\end{lemma}

\begin{proof}
 Given $F\in \Hom(G,H)$, the naturality of the exponential map gives commutativity of the right square in the diagram
\begin{equation*}
    \begin{tikzcd}
0 \arrow[r] & \fgz \arrow[r, hook] \arrow[d, "\dd F_e\vert_{\fgz}", dashed] & \fg \arrow[r, "\exp_G", two heads] \arrow[d, "\dd F_e"] & G \arrow[r] \arrow[d, "F"] & 0 \\
0 \arrow[r] & \fhz \arrow[r, hook]                                                & \fh \arrow[r, "\exp_H", two heads]                            & H \arrow[r]                      & 0.
\end{tikzcd}
\end{equation*}
Since $\dd F_e[\fgz] \subseteq \fhz$, the map $\dd F_e\vert_{\fgz}$ making the left square commute is well defined.

Conversely, given $f_{_\ZZ}\in \Hom_{_\ZZ}(\fgz,\fhz)$, there is a unique linear map $f:\fg \to \fh$ such that the left square in the diagram 
\begin{equation*}
    \begin{tikzcd}
0 \arrow[r] & \fgz \arrow[r, hook] \arrow[d, "f_{_\ZZ}"] & \fg \arrow[r, "\exp_G", two heads] \arrow[d, "f"] & G \arrow[r] \arrow[d, "F", dashed] & 0 \\
0 \arrow[r] & \fhz \arrow[r, hook]                          & \fh \arrow[r, "\exp_H", two heads]                      & H \arrow[r]                              & 0
\end{tikzcd}
\end{equation*}
commutes.
Since $\exp_H(f[\fgz])=0$, there exists a unique map $F:G\to H$ making the right square commute.
\end{proof}

\begin{example}
\label{ex:integrallattice}
    For $G = S^1 = \{z \in \CC\mid \abs{z}=1\}$ with Lie algebra $\RR$ and exponential map
    $t\mapsto e^{2\pi i t}$, the integral lattice is $\ZZ$. Then \eqref{eq:hom_correspondence} yields a correspondence between periodic one-parameter subgroups of $H=T$ and $\ftz$:
\begin{equation*}
    \Hom(S^1,T) \cong \Hom_{_\ZZ}(\ZZ,\ftz) \cong \ftz.
\end{equation*}
\end{example}

Let $\ft^*$ be the dual of $\ft$. Write the natural pairing as: 
\begin{align*}
    \langle \cdot , \cdot \rangle :
   \ft^* \times \ft &\longrightarrow \RR \\
   (\xi,X)&\longmapsto \langle \xi , X \rangle := \xi (X),
\end{align*}
and denote by
\[
\ft_{_{\ZZ}}^* := {\rm{Hom}}_{_{\ZZ}} (\ft_{_{\ZZ}}, \ZZ)
\]
the dual of the integral lattice. Using the identification
\[
\ft_{_{\ZZ}}^* \cong \{ \xi \in \ft^* \mid \xi ({\ftz}) \subseteq \ZZ\},
\]
this can be viewed as a lattice in $\ft^*$.
We refer to $\ft_{_{\ZZ}}^*$ as the \textit{weight lattice}
of $T$ because of the following isomorphism to the
character group.

\begin{example} \label{ex:charactergroup}
    For $H= S^1$, the correspondence \eqref{eq:hom_correspondence}
    yields an isomorphism between the \textit{character group} $\widehat{T}=\Hom(T,S^1)$ and the weight lattice $\ftz^*$ as
\begin{equation*}
    \Hom(T,S^1) \cong \Hom_{_\ZZ}(\ftz,\ZZ) = \ftz^*.
\end{equation*}
For later convenience, we register here that
the inverse is given explicitly by
\begin{align}
    \Hom_{_\ZZ}(\ftz,\ZZ) & \; \longrightarrow \; \Hom(T,S^1) \nonumber \\
    \lambda:X \mapsto \expval{\lambda,X} & \; \longmapsto \; \chi_\lambda:\exp(X) \mapsto  e^{2\pi i \expval{\lambda, X}} \nonumber 
\end{align}
\end{example}


A \textit{lattice basis} (or simply \textit{basis}) of $\ftz$
is a basis of $\ft$ that generates $\ftz$ over $\ZZ$;
such a basis always exists~\cite[Thm 10.4]{barvinok2008integer}.
Given a lattice basis $X_1, \ldots , X_n$ of $\ft_{_{\ZZ}}$,
the corresponding dual elements $\lambda_j : \ft_{_{\ZZ}} \to \ZZ$,
satisfying $\lambda_j (X_i) = \delta_{ij}$, form a lattice basis of
$\ft_{_{\ZZ}}^*$. Any lattice basis $X_i$ of the integral lattice determines a splitting of the torus into $S^1$-subgroups, where the $i$-th subgroup is generated by $X_i$ via the correspondence of \cref{ex:integrallattice}.


\subsection{Closed Subgroups}
\label{sec:closed}

Let $G$ be a Lie subgroup of the torus $T$, i.e., a subgroup
such that the inclusion map is a smooth injective immersion.
Let $\fg$ be the corresponding Lie subalgebra of $\ft$. In general, the discrete subgroup $\fgz := \fg \cap  \ftz$ is not a lattice in $\fg$ as it might fail to span $\fg$.
We say that $\fg$ is \textit{rational} in $\ft$ when $\fgz = \fg \cap \ftz$ is a lattice in $\fg$.
The following well-known fact is proven, for instance, in~\cite[p.61]{duistermaat2000lie}.
\begin{proposition}
    A Lie subgroup $G$ of a torus $T$ is a closed subgroup if and only if its Lie algebra $\fg$ is rational in the Lie algebra of the torus $\ft$ and $G$ has a finite number of connected components.
\end{proposition}

\begin{example}
    Consider a two-dimensional torus $T$. Given a basis $X_1,X_2$ of $\ftz$, set $Y:= X_1 + \alpha X_2$ for $\alpha\in \RR$.
    The one-parameter subgroup generated by $Y$ is closed if and only if
    the subspace spanned by $Y$ is rational in $\ft$, and that is
    if and only if $\alpha\in \mathbb{Q}$.
\end{example}

If $G\subseteq T$ is closed, the short exact sequence of abelian Lie groups
\[
\begin{tikzcd}
0 \arrow{r} 
  & G \arrow[hookrightarrow]{r}
  & T \arrow[two heads]{r}
  & T/G \arrow{r}
  & 0
\end{tikzcd}
\]
gives rise to an exact sequence of Lie algebras 
\begin{equation} \begin{tikzcd}
    0 \arrow{r} 
  & \fg \arrow[hookrightarrow]{r}{i}
  & \ft \arrow[two heads]{r}{\pi}
  & \Lie(T/G) \arrow{r}
  & 0.
  \end{tikzcd}
\end{equation}
$\Lie(T/G)$ is thus canonically isomorphic to $\ft/\fg$ via the differential of the quotient map $T \twoheadrightarrow T/G$ and we will identify them with each other.


\subsection{The Annihilator Mechanism}
\label{sec:annihilator}

From the standard annihilator mechanism (\cref{thm:annihilatormechanism}) composed with the isomorphism $\widehat{T} \cong \ftz^*$ from \cref{ex:charactergroup} we deduce the following correspondence.
Note that, since $\ftz^*$ is discrete, the subgroups of $\ftz^*$ are automatically closed.

\begin{theorem}[The Annihilator Mechanism for Tori]
\label{thm:annihilatorfortorus}
Let $T$ be a torus with Lie algebra $\ft$ and weight lattice $\ft_{\ZZ}^* \subset \ft^*$. The functions
\begin{align*}
    \{\text{closed subgroups of } T\} &\longleftrightarrow \{\text{subgroups of } \ft_{\ZZ}^*\} \\
    H &\longmapsto \Lambda_H := \{\lambda \in \ft_{\ZZ}^* \mid \chi_\lambda\vert_H = 1 \}\\
    \bigcap_{\lambda\in \Lambda}\ker(\chi_\lambda)=:H_\Lambda&\mapsfrom \Lambda
\end{align*}
are mutually inverse and, for any two closed subgroups $H_1,H_2\subset T$, we have
\begin{itemize}
    \item $H_1\subset H_2$ if and only if $\Lambda_{H_2}\subset \Lambda_{H_1}$,
    \item $\Lambda_{H_1\cap H_2}=\Lambda_{H_1}+\Lambda_{H_2}$ and
    \item $\Lambda_{H_1+H_2}=\Lambda_{H_1}\cap\Lambda_{H_2}$.
\end{itemize}
\end{theorem}

\begin{remark}
Since the weight lattice $\ftz^*$ is a finitely generated abelian group, every subgroup $\Lambda \subset \ftz^*$ is also finitely generated. Consequently, if $\{\lambda_1, \dots, \lambda_k\}$ generate $\Lambda$, $H_\Lambda$ is given by a finite intersection:
\[
H_\Lambda = \bigcap_{i=1}^k \ker(\chi_{\lambda_i}).
\]
\end{remark}

\begin{proposition}[Geometric Structure from Annihilator]\label{prop:geometricstructure}
    Let $T$ be a torus, and let $\ftz^*\subset \ft^*$ be the weight lattice. Let $H\subset T$ be a closed subgroup, let $H_0\subset H$ be the identity component, and let $\fh^0\subset \ft^*$ denote the annihilator of $\fh = \Lie(H)$. Then:
    \begin{enumerate}
    \item $\fh^0=\spn{\RR}{\Lambda_H}$,
    \item $\fh^0\cap \ftz^*=\Lambda_{H_0}$ and
    \item $H/H_0 \cong \Lambda_{H_0}/\Lambda_H.$
\end{enumerate}
\end{proposition}

\begin{proof}
    \begin{enumerate}
        \item We have
        \begin{align*}
            X\in \fh &\iff \exp(tX)\in H \quad \forall t\in \RR \\
            &\iff \exp(tX)\in \ker(\chi_\lambda) \quad \forall t\in \RR, \,\forall \lambda\in \Lambda_H\\
            &\iff \expval{\lambda,tX}\in\ZZ \quad \forall t\in \RR, \, \forall \lambda \in \Lambda_H \\
            &\iff \expval{\lambda,X}=0 \quad \forall\lambda \in \Lambda_H \\
            &\iff X\in \spn{\RR}{\Lambda_H}^0.
        \end{align*}
        \item Since $\Lie(H_0) = \Lie(H)$, we have $\spn{\RR}{\Lambda_{H_0}}=\fh^0$ so $\Lambda_{H_0} \subseteq \fh^0 \cap \ftz^*$. Conversely, let $\lambda \in \fh^0\cap \ftz^*$. Since $H_0$ is a subtorus, $H_0 = \exp_T(\fh)$.
        Then \begin{equation*}
            \chi_\lambda(\exp_T(X)) = e^{2\pi i\expval{\lambda,X}} =1
             \quad \forall X\in \fh,
        \end{equation*}
        which implies that $\chi_\lambda\vert_{H_0}=1,$ that is,
        $\lambda\in \Lambda_{H_0}$.
        \item Since $H/H_0$ is a finite group, it is (non-canonically) isomorphic to its character group. Hence, we have
        \begin{equation*}
            H/H_0 \cong \widehat{H/H_0} \cong \Lambda_{H_0}/\Lambda_H. \qedhere
        \end{equation*}
    \end{enumerate}
\end{proof}

\cref{table:dictionary} gives an overview of the torus-lattice correspondence from the annihilator mechanism. 

\renewcommand{\arraystretch}{1.33}
\begin{table}[h]
\centering
\begin{tabular}{|l||c|c||c|} \cline{2-4}
     \multicolumn{1}{l||}{}                                & \textbf{Torus Side}                                               & \textbf{Lattice Side}                                                   &  Statement                                                                 \\ \hline\hline
\multirow{2}{*}{Definitions}         & $H$                                                      & $\{\lambda\in \ftz^*\mid \chi_\lambda\vert_H = 1\}$ & \multirow{2}{*}{{\cref{thm:annihilatorfortorus}}} \\
                                     & $\cap_{\lambda\in \Lambda}\ker(\chi_\lambda)$ & $\Lambda$                                                      &                                                                   \\ \hline
\multirow{3}{*}{Basic Operations}    & $H_1\subset H_2$                                         & $\Lambda_{H_2}\subset \Lambda_{H_1}$                           & \multirow{3}{*}{\cref{thm:annihilatorfortorus}} \\
                                     & $H_1\cap H_2$                                            & $\Lambda_{H_1} + \Lambda_{H_2}$                                &                                                                   \\
                                     & $H_1 + H_2$                                              & $\Lambda_{H_1}\cap \Lambda_{H_2}$                              &                                                                   \\ \hline
\multirow{3}{*}{Geometric Structure} & $\Lie(H)$                                                & $\fh^0 =\spn{\RR}{\Lambda_H}$                                  & \multirow{3}{*}{\cref{prop:geometricstructure}}  \\
                                     & $H_0$                                                    & $\Lambda_{H_0}=\fh^0\cap \ftz^*$                &                                                                   \\
                                     & $H/H_0$                                                  & $\Lambda_{H_0}/\Lambda_H$                                      &                                                                   \\ \hline
\multirow{2}{*}{Interpretation}      & $\widehat{H}$                                            & $\ftz^*/\Lambda_H$                                             & \multirow{2}{*}{\cref{prop:interpretation}}      \\
                                     & $\widehat{T/H}$                                          & $\Lambda_H$                                                    & \\ \hline                                                                 
\end{tabular}
\vspace{.5cm}
\caption{Dictionary for the annihilator mechanism and a closed subgroup $H\subset T$ with identity component $H_0.$}
\label{table:dictionary}
\end{table}
\begin{example}[Disconnectedness from Non-primitive Generators]
Consider the circle $T = S^1$ with weight lattice $\ftz^* \cong \ZZ$. For $n \in \NN_{>1}$, consider the subgroup $\Lambda = n\ZZ \subset \ftz^*$, generated by the non-primitive element $n$. By definition, $$H_\Lambda = \ker(\chi_n) = \{z \in S^1 \mid z^n = 1\} \cong \ZZ_n.$$
\end{example}
\begin{example}[Disconnectedness from Primitive Generators]
Let $T = S^1 \times S^1$ with weight lattice $\ftz^* \cong \ZZ^2$. Consider $\Lambda = \spn{\ZZ}{(1,1), (-1,1)} \subset \ftz^*$, generated by the primitive vectors $\lambda_1 = (1,1)$ and $\lambda_2 = (-1,1)$. 
By definition, 
\begin{equation*}
    H_\Lambda = \ker(\chi_{(1,1)}) \cap \ker(\chi_{(-1,1)}) = \{(z_1, z_2) \in T \mid z_1 z_2 = 1 \text{ and } z_1^{-1} z_2 = 1\},
\end{equation*}
that is,  $H = \{(1, 1), (-1, -1)\}$.
\end{example}

In \cref{sec:toric_reps,sec:reduction_1}, we will have to understand intersections of subtori. In view of \cref{prop:geometricstructure}, these allow for particularly simple algebraic counterparts.
\begin{corollary}[Comparing Subtorus Intersections]
\label{cor:subtorusintersection}
Let $G_1, G_2, H_1, H_2\subset T$ be subtori. Then 
\begin{equation*}
    G_1\cap G_2 = H_1\cap H_2
\end{equation*}
if and only if
\begin{equation*}
    (\fg_1^0\cap \ftz^*) + (\fg_2^0 \cap \ftz^*) = (\fh_1^0 \cap \ftz^*) + (\fh_2^0\cap \ftz^*).
\end{equation*}
\end{corollary}

\begin{proof}
    By \cref{thm:annihilatorfortorus}, $G_1\cap G_2 = H_1\cap H_2$ is equivalent to $\Lambda_{G_1\cap G_2}=\Lambda_{H_1\cap H_2}$. But since $G_1$ and $G_2$ are subtori, we have
\begin{equation*}
    \Lambda_{G_1\cap G_2} = \Lambda_{G_1} +\Lambda_{G_2} = (\fg_1^0\cap \ftz^*) + (\fg_2^0 \cap \ftz^*)
\end{equation*}
and analogously for $\Lambda_{H_1\cap H_2}$.
\end{proof} 

\begin{corollary}[Component Group of Subtorus Intersection]
Let $G,H\subset T$ be subtori. Then 
\begin{equation*}
    \frac{G\cap H}{(G\cap H)_0} \cong \frac{(\fg^0 + \fh^0)\cap \ftz^*}{(\fg^0 \cap \ftz^*)+(\fh^0 \cap \ftz^*)}.
\end{equation*}
\end{corollary}
\begin{proof}
    By \cref{prop:geometricstructure}, 
\begin{equation*}
    \frac{G\cap H}{(G\cap H)_0} \cong \frac{\Lambda_{(G\cap H)_0}}{\Lambda_{G\cap H}},
\end{equation*}
where $\Lambda_{G\cap H}$ is computed as above. Since $(G\cap H)_0$ is a subtorus, we also have
\begin{equation*}
    \Lambda_{(G\cap H)_0} =  (\fg\cap \fh)^ 0 \cap \ftz^* = (\fg^0 + \fh^0)\cap \ftz^*. \qedhere
\end{equation*}
\end{proof}

\begin{example}Consider the two-torus $T=S^1 \times S^1$ with weight lattice $\ftz^* \cong \ZZ^2$. Let $G=S^1 \times \{1\}$ and $H=\exp(\RR Y)$ where $Y = (m,n) \in \ZZ^2$ is primitive with $n \neq 0$. We have $\fg^0 = \RR (0,1)$ and $\fh^0 = \RR (-n,m)$, so that
$$\Lambda_G=\fg^0 \cap \ftz^* = \ZZ(0,1)\qq{and}\Lambda_H=\fh^0 \cap \ftz^* = \ZZ(-n,m).$$
Since $(\fg\cap \fh)^0 =\fg^0 + \fh^0 = \ft^*$, we also get
$$\Lambda_{(G\cap H)_0}=(\fg^0 + \fh^0) \cap \ftz^* \cong \ZZ^2.$$ The quotient is therefore: $$\frac{G\cap H}{(G\cap H)_0} \cong \frac{\ZZ^2}{\ZZ(0,1) \oplus \ZZ(-n,m)} \cong \ZZ_{\abs{n}}.$$
Since the intersection is discrete, we actually conclude that $G \cap H \cong \ZZ_{\abs{n}}$.
\end{example}

\begin{remark}
\label{rmk:integral_transv_origin}
The intersection $G \cap H$ is connected exactly
when $(\fg^0+\fh^0)\cap \ftz^* = (\fg^0\cap \ftz^*)+(\fh^0\cap \ftz^*)$.
When $H$ is $1$-dimensional and $\ft = \fg \oplus \fh$, this reduces to the
requirement that $\fg^0$ and $\fh^0$ intersect
\textit{integrally transversely} in the sense of
McDuff~\cite{mcduff2011displacing}. A substantial generalization of this
notion -- to arbitrary polytopes and affine subspaces, and without
requiring $\ft = \fg \oplus \fh$ -- underlies the treatment of reduction
throughout \Cref{sec:reduction_1,sec:reduction_2}.
\end{remark}


\section{The Toric Representation Framework}
\label{sec:toric_reps}

In this section, we apply the annihilator mechanism to the study of symplectic torus representations and establish that the kernel of any such representation is dual to the integer span of its weights (\cref{sec:decomposability}).
We then specialize to \textit{toric} representations, which are torus representations whose kernel is connected and satisfies the dimension requirement of a toric action (\cref{def:s_toric_rep}). The annihilator mechanism immediately translates these conditions into the geometric requirement that the associated moment cone be unimodular (\cref{sec:representations}). This yields simple linear counterparts to Delzant's Theorem and Delzant's Lemma, as well as a classification of toric symplectic subrepresentations via the faces of the cone (\cref{sec:toric_subrep}). Together, these results form the representation-theoretic groundwork underlying the toric local models utilized in \cref{sec:stm_framework}.


\subsection{Symplectic Torus Representations}
\label{sec:decomposability}

Let $G$ be a torus with Lie algebra $\fg$, and let $\fgz^* \subset \fg^*$ denote its integral weight lattice. Let $\operatorname{Sp}(V,\omega)$ denote the Lie group of symplectic isomorphisms of a finite-dimensional real symplectic vector space $(V,\omega)$.

\begin{definition}
    A \textit{symplectic $G$-representation} consists of a symplectic vector space $(V,\omega)$ and a smooth homomorphism $\rho:G\to \operatorname{Sp}(V,\omega)$. 
\end{definition}

\begin{definition}
Two symplectic $G$-representations, $\rho_1:G \to \operatorname{Sp}(V_1,\omega_1)$
and $\rho_2:G \to \operatorname{Sp}(V_2,\omega_2)$, are \textit{isomorphic} 
if there exists a $G$-equivariant symplectic isomorphism
$\varphi : V_1 \to V_2$, i.e., a linear isomorphism such that
$\varphi^* \omega_2 = \omega_1$ and
$\varphi (\rho_1(g) v) = \rho_2(g) \varphi (v)$
for all $g \in G$, $v \in V_1$.
\end{definition}

\begin{lemma}[{\cite[Lemma A.1]{lerman1997hamiltonian}}]
    \label{lem:classificationtorusreps}
There is a bijective correspondence between isomorphism classes of $2n$-dimensional symplectic $G$-representations and unordered $n$-tuples of elements (possibly with repetition) of the weight lattice $\fgz^* \subset \fg^*$.

\medskip

Let $(V,\omega)$ be the $2n$-dimensional symplectic $G$-representation with weights  $\lambda_1, \ldots, \lambda_n \in \fgz^*$. There exists a decomposition 
$$(V, \omega) = \bigoplus_{i=1}^n (V_i, \omega_i)$$ 
into invariant, mutually $\omega$-orthogonal 2-dimensional symplectic subspaces, and a $G$-invariant norm $\abs{\cdot}$ induced by a compatible complex structure, such that:
\begin{enumerate}
    \item By identifying each $(V_i, \omega_i, \abs{\cdot})$ with $\mathbb{C}$, the representation of $G$ on $V_i$ is given by complex multiplication via the character $\chi_{\lambda_i}:G\to S^1$.
    \item The canonical moment map for the $G$-action is given by
\begin{align*}
    \mu:V=\oplus_{i=1}^n V_i &\longrightarrow \fg^* \\
    (v_1,\dots,v_n) &\longmapsto \frac{1}{2} \sum_{i=1}^n \abs{v_i}^2\lambda_i
\end{align*}
\end{enumerate}
\end{lemma}

\begin{remark}
Recall that moment maps are only determined up to an additive constant. By \emph{canonical} we   mean that $\mu(0) = 0.$
\end{remark}

\subsection{Classification by Moment Cones} 
\label{sec:representations}

Let $G$ be a torus.

\begin{definition}
\label{def:s_toric_rep}
A \textit{toric symplectic $G$-representation} is a symplectic representation $\rho:G \to \text{Sp}(V,\omega)$ on a symplectic vector space $(V,\omega)$ such that the kernel $G_\ker = \opker(\rho)$ is a subtorus of dimension $\dims{G}-\frac{1}{2}\dims{V}$.\footnote{For consistency, we also view the emtpy set as a toric symplectic $G$-representation.}
\end{definition}

It follows from \cref{lem:classificationtorusreps} that the image under the moment map of any symplectic torus representation $\rho:G \to \sp(V,\omega)$ is a polyhedral cone, namely the conic hull of its weights $\lambda_1,\dots,\lambda_n\in\fgz^*$: 
\begin{equation*}
    C:=\mu(V) = \co{\lambda_1,\dots,\lambda_n}.
\end{equation*}
It is called the \textit{moment cone} of the representation $\rho$. 

We now apply the annihilator mechanism to reformulate the conditions for a representation to be toric in terms of its weights, respectively its moment cone. The starting-point is the following, almost tautological observation:

\begin{lemma}[Kernel of Symplectic Torus Representation] 
    Let $\rho: G \to \text{Sp}(V, \omega)$ be the symplectic $G$-representation with weights $\lambda_1,\dots,\lambda_n\in \fgz^*$. Then
    \begin{equation*}
        \Lambda_{\ker(\rho)} = \spn{\ZZ}{\lambda_1,\dots,\lambda_n}.
    \end{equation*}
\end{lemma}
\begin{proof}
    This is essentially by definition as
    \begin{equation*}
        \ker(\rho) = \bigcap_{i=1}^n\ker(\chi_{\lambda_i})=H_{\spn{\ZZ}{\lambda_1,\dots,\lambda_n}}. \qedhere
    \end{equation*}
\end{proof}

Combining this with \cref{prop:geometricstructure}, yields that the following are equivalent:
\begin{itemize}
    \item $\dim(G_\ker) =\dim(G) - \tfrac{1}{2}\dim(V)$
    \item $\dim(\spn{\RR}{\lambda_1,\dots,\lambda_n})=\dim(\fg_\ker^0) = \tfrac{1}{2}\dim(V) = n.$
\end{itemize}
The dimension requirement is thus equivalent to the weights $\lambda_1,\dots,\lambda_n\in\fgz^*$ being linearly independent over $\RR$. This is precisly the condition for the moment cone $\co{\lambda_1,\dots,\lambda_n}$ to be a \emph{simple cone}.

In the same way, we see that $G_\ker$ is connected if and only if
\begin{equation*}
    \spn{\RR}{\lambda_1,\dots,\lambda_n}\cap \fgz^*=\spn{\ZZ}{\lambda_1,\dots,\lambda_n},
\end{equation*}
which is the condition for $\co{\lambda_1,\dots,\lambda_n}$ to be \emph{unimodular} (see \cref{def:unimodular_cone}). 

Conversely, given any unimodular cone $C$ in $\fg^*$, there is a unique toric symplectic representation $\rho$ determined by the weights forming the edges of the cone. By construction, the moment cone of this representation is precisely $C$. We thus have the following elementary linear counterpart of the Delzant Theorem:
\begin{proposition}[Delzant Classification for Toric Representations]
Toric symplectic $G$-representations are classified, up to isomorphism,
by unimodular cones in $\fg^*$, where the toric symplectic representation
with weights $\lambda_1,\dots,\lambda_n\in\fgz^*$ is
associated with the conic hull of $\lambda_1,\dots,\lambda_n$.
\end{proposition}

\medskip

Let $V=\bigoplus_{i=1}^n V_i$ be the decomposition of a toric symplectic $G$-representation into weight spaces. For any vector $v = (v_1, \dots, v_n) \in V$, we can recover the geometric structure of the stabilizer $G_v\subset G$ by substituting the full set of weights for the subset $\{\lambda_i \mid v_i \neq 0\}$ in the above arguments. This yields the following linear counterpart of Delzant's Lemma:
\begin{lemma}[Delzant's Lemma for Toric Representations]
\label{lem:representationdelzant}
Let $\rho:G\to \text{Sp}(V,\omega)$ be the toric symplectic $G$-representation with moment map $\mu$ as in \cref{lem:classificationtorusreps}
and moment cone $C\subseteq \fg^*$. Then 
\begin{enumerate}
    \item $\mu:V\to \fg^*$ is the point-orbit map;
    \item for $\xi\in C$, the fiber $\mu^{-1}(\xi)$ is a torus with $\dim(\mu^{-1}(\xi)) = \dim(L_\xi C)$;
    \item for $v\in V$, the stabilizer $G_v\subseteq G$ is the subtorus with $\fg_v^0=L_{\mu(v)}C$.
\end{enumerate}
\end{lemma}

\begin{proof}
    Let $\lambda_1,\dots,\lambda_n\in\fgz^*$ be the weights of $\rho$ so that the moment cone is
\begin{equation*}
    C=\co{\lambda_1,\dots,\lambda_n}.
\end{equation*}
For any $\xi\in C$, there are unique nonnegative real coefficients $\alpha_j$ such that
\begin{equation*}
    \xi = \sum_{j=1}^n \alpha_j \lambda_j
\end{equation*}
Using the explicit form of the moment map given in \cref{lem:classificationtorusreps}, it follows that
\begin{equation*}
     \mu^{-1}(\xi) = S_1 \times \dots \times S_n
     \subseteq V = \bigoplus_{j=1}^n V_j
\end{equation*}
where $S_j = \{ v_j\in V_j \mid \abs{v_j} = \sqrt{2\alpha_j}\}.$
\begin{enumerate}
\item[1.] We have $\chi_{\lambda_j}(G) = S^1$ since every nontrivial character $\chi:G\to S^1$ is surjective and $\lambda_j\neq 0$ for all $j$. It follows that $G\cdot v_j = S_j$. Linear independence of $\lambda_1,\dots,\lambda_n$ allows to rotate independently in each factor $V_j$, giving $G\cdot (v_1,\dots,v_n)=S_1\times \dots\times S_n=\mu^{-1}(\xi)$.
\item[3.] We have $\Lambda_{G_v}=\spn{\ZZ}{\{\lambda_i \mid v_i \neq 0\}}$. It follows immediately from \cref{prop:geometricstructure} that $G_v$ is connected and $\fg_v^0 = \spn{\RR}{\{\lambda_i \mid v_i \neq 0\}}=L_{\mu(v)}C$. 
\item[2.] Follows from 1. and 3. or the fact that for $j\in \{1,\dots,n\}$ such that $\alpha_j =0$, the circle $S_j$ reduces to a point. \qedhere
\end{enumerate}
\end{proof}


\subsection{Symplectic Toric Subrepresentations}
\label{sec:toric_subrep}

\begin{definition}
A \textit{toric symplectic subrepresentation} of a toric symplectic $G$-representation $(V,\omega)$ is a $G$-invariant symplectic subspace $W\subseteq V$.
\end{definition}

\begin{proposition} \label{prop:representationdelzant}
Let $\rho:G\to \text{Sp}(V,\omega)$ be a toric symplectic $G$-representation. Then
\begin{enumerate}
    \item the toric symplectic subrepresentations of $\rho:G\to \text{Sp}(V,\omega)$ are exactly the preimages of the faces of its moment cone $C_V$, and
    \item any toric symplectic subrepresentation is itself a toric symplectic $G$-representation.
\end{enumerate}
\end{proposition}

\begin{proof}
  The first point follows directly from \cref{lem:classificationtorusreps} and the observation that symplectic subrepresentations are classified by subsets of weights. The second point follows from the linear independence of the weights,
  since any subset is still linearly independent and $\ZZ$-spanning.
\end{proof}


\section{The Toric Manifold Framework}
\label{sec:stm_framework}

We now move from representations to manifolds.
Recall that we allow our torus
actions to have a nontrivial kernel throughout the paper, so that
every construction of \cref{table:transformations_summary} can be
carried out within a single copy of $\fg^*$, regardless of its own
dimension.
We begin by constructing the toric local model
(\cref{sec:local_model}), which immediately feeds into the Symplectic Slice Theorem
(\cref{sec:slice_thm}).
We show that a toric manifold's
local models are themselves toric and compute the kernel of its
slice representations.
With this local control in hand, we define
symplectic toric $G$-manifolds with possibly nontrivial connected kernels
(\cref{sec:nontrivial_kernels}) and extend Delzant's classification
theorem to this setting (\cref{sec:classification}),
before showing, in \cref{sec:submanifolds},
that the symplectic toric submanifolds of such a manifold are exactly
the preimages of the faces of its moment polytope -- the manifold-side
counterpart of ``taking a face'' from \cref{table:transformations_summary},
and the source of the toric submanifolds that reappear throughout
\cref{sec:reduction_2}.


\subsection{Toric Local Model}
\label{sec:local_model}

The local structure of a Hamiltonian $G$-space near an orbit is
a main support for the work in this paper.
We recall here the construction of the model spaces to which every such neighborhood is equivariantly symplectomorphic, culminating in Delzant's Lemma for Local Models (\cref{lem:delzantlocalmodel}).
This is the technical input for the Symplectic Slice Theorem below,
and we use it again for the classification of moment polytopes (\cref{sec:nontrivial_kernels}) and for symplectic toric submanifolds (\cref{sec:submanifolds}).
The classification of toric symplectic representations on which the lemma relies is recalled in \cref{sec:toric_reps}.

Left-trivialization allows us to write the cotangent bundle of the torus $G$ as $G\times \fg^*$ and gives, for any point $(g,\xi)\in G\times \fg^*$, the identification $T_{(g,\xi)}(G\times \fg^*) \cong \fg \times \fg^*$. Using this, the canonical symplectic form is given by 
\begin{equation*}
    \omega_{(g,\xi)}\left((X,\alpha),(Y,\beta) \right) = \alpha(Y) - \beta(X)
\end{equation*}
for $(X,\alpha),(Y,\beta)\in \fg \times \fg^*$.
The cotangent lift of left translation\footnote{Since $G$ is abelian, there is no difference between left and right translation. We use this terminology and notation here purely for clarity of the exposition.} by $h\in G$ is
\[
h\cdot (g,\xi)=(hg,\xi)
\]
and defines a Hamiltonian group action with moment map $(g,\xi)\mapsto \xi$.
On the other hand, the cotangent lift of right translation by the inverse,
\[
h\cdot (g,\xi) = (gh^{-1},\xi),
\]
defines a Hamiltonian $G$-action with moment map $(g,\xi)\mapsto-\xi$.

Given a symplectic representation $\rho:H\to \sp(V,\omega_{_V})$ of a closed subgroup $H\subseteq G$, consider the diagonal action of $H$ on the product $G\times \fg^* \times V$, where the action on $G\times \fg^*$ is by the cotangent lift of right translation, and the action on $V$ is by the representation $\rho$. This action is free and Hamiltonian with moment map
\[
\mu_{_H} : G\times \fg^* \times V \longrightarrow \fh^*,
\quad \mu_{_H}(g,\xi,v)= -i^*(\xi) + \mu_{_V}(v),
\]
where $i^*:\fg^* \to \fh^*$ is the canonical projection and
$\mu_{_V}:V\to \fh^*$ is the moment map of $\rho$ satisfying $\mu_{_V}(0)=0$ as in
\cref{lem:classificationtorusreps}.
We denote the symplectic quotient by
\[
E:=\mu_{_H}^{-1}(0)/H,
\]
and its reduced symplectic form $\omega_{_E}$.

The action of $G$ on $G\times \fg^*$ by the cotangent lift of left translation can be extended to $G\times \fg^* \times V$ by acting trivially on $V$. This action preserves $\mu_{_H}^{-1}(0)$ and commutes with the $H$-action. Hence, it descends to a Hamiltonian action on the quotient $E$, whose moment map is $\mu_{_E}([g,\xi,v])=\xi$.

\begin{definition}
\label{def:localmodel}
    The triple $(E,\omega_{_E},\mu_{_E})$, or simply $E$, with the $G$-action induced by left translation is called the \textit{local model} (or \textit{symplectic local model}) defined by the symplectic representation $\rho:H \to \sp(V,\omega_{_V})$.
    In the special case where $H$ is a subtorus of $G$ and $\rho:H\to \sp(V,\omega)$ is a toric symplectic representation, we say that $E$ is a \textit{toric local model.}
\end{definition}

\begin{remark}
     If $H\subseteq G$ is a subtorus, we can find a subtorus $K\subseteq G$ that is complementary to $H$. Then the local model simplifies to $E=K\times \mathfrak{k}^*\times V$ and $G=K\times H$ acts diagonally with moment map
     \[
     \mu_{_E}:K\times \mathfrak{k}^* \times V \longrightarrow \mathfrak{k}^*\oplus \fh^*,\qquad (k,\varphi,v)\longmapsto (\varphi,\mu_{_V}(v)).
     \]
     This depends upon the choice of the complementary subtorus $K\subset G$.
\end{remark}

\begin{lemma}[Basic Properties of the Local Model]\label{lem:local_model_properties}
Let $H\subseteq G$ be a closed subgroup, and let $E$ be the local model defined by a symplectic $H$-representation $\rho:H \to \sp(V,\omega_{_V})$. Then:
\begin{enumerate}
    \item For any point $[g,\xi,v]\in E$, the stabilizer in $G$ is
    \[
    G_{[g,\xi,v]} = H_v.
    \]
    \item The image of the moment map is
    \[
    \mu_{_E}(E) = (i^*)^{-1}(C),
    \]
    where $i^*:\fg^* \to \fh^*$ is the canonical projection and $C$ is the moment cone of $\rho$.
\end{enumerate}
\end{lemma}
\begin{proof}
\begin{enumerate}
    \item If $g'\in H_v$, then
    \begin{align*}
        g'\cdot [g,\xi,v] &= [g'g,\xi,v] & \text{by definition of the $G$-action}\\
        &= [gg',\xi,v] & \text{since $G$ is abelian}\\
        &= [g,\xi,\rho_{g'}(v)] & \text{using the $H$-action} \\
        &= [g,\xi,v] & \text{because $g'\in H_v$.}
    \end{align*}Conversely, if $g'\in G_{[g,\xi,v]}$ we have $[g,\xi,v] = [g'g,\xi,v]$.Hence, there is an $h\in H$ such that\begin{equation*}(g'g,\xi,v)=h\cdot (g,\xi,v) = (gh^{-1},\xi,\rho_h(v)),\end{equation*}implying that $g'=h^{-1}\in H$ and $v = \rho_h(v)$. Since $h \in H_v$ and $H_v$ is a subgroup, we conclude $g' = h^{-1} \in H_v$.
    \item By definition, $[g,\xi,v] \in E$ if and only if $\mu_{_H}(g,\xi,v) = -i^*(\xi) + \mu_{_V}(v) = 0$. Since $\mu_{_E}([g,\xi,v]) = \xi$, a point $\xi \in \fg^*$ is in the image $\mu_{_E}(E)$ if and only if there exists a $v \in V$ such that $i^*(\xi) = \mu_{_V}(v)$. \qedhere\end{enumerate}\end{proof}

\begin{lemma}[Delzant's Lemma for Toric Local Models]
\label{lem:delzantlocalmodel}
    Let $H\subseteq G$ be a subtorus, and let $E$ be the toric local model defined by $\rho:H\to \sp(V,\omega)$. If $P=\mu_{_E}(E)$ denotes the image of the moment map, then
    \begin{enumerate}
        \item $\mu_{_E}:E \to \fg^*$ is the point-orbit map;
        \item for $\xi \in P$, the fiber $\mu_{_E}^{-1}(\xi)$ is a torus
        with $\dim(\mu_{_E}^{-1}(\xi)) = \dim(L_\xi P)$;
        \item for $q\in E$, the stabilizer $G_q \subseteq G$ is the subtorus with $\fg_q^0 = L_{\mu_{_{_E}}(q)}P$.
    \end{enumerate}
\end{lemma}

\begin{proof}
\begin{enumerate}
    \item[1.] We show that $G$ acts transitively on $\mu_{_E}^{-1}(\xi)$:
    Given two elements $[t_1,\xi, v_1], [t_2,\xi, v_2]\in \mu_{_E}^{-1}(\xi)$, observe that $v_1,v_2\in \mu_{_V}^{-1}(i^*(\xi))$.
    By \cref{lem:representationdelzant}, there is $g\in H$ such that $\rho_{g}(v_1) = v_2$.
    Then $h = t_2gt_1^{-1}$ satisfies
    \begin{align*}
        h \cdot [t_1,\xi,v_1]&= [h t_1,\xi,v_1] \\
        &= [t_2g,\xi,v_1] \\
        &= [t_2,\xi,\rho_{g}(v_1)] \\
        &= [t_2,\xi,v_2].
    \end{align*}
    \item[3.] By \cref{lem:representationdelzant}, $H_v\subseteq H$ is a subtorus.
    Hence, by \cref{lem:local_model_properties},
    the stabilizer $G_{[g,\xi,v]}\subseteq G$ is also a subtorus.
    The Lie algebra of $G_{[g,\xi,v]}$ satisfies
    \begin{align*}
        \fg_{[g,\xi,v]} &= i(\fh_v) & \text{by \cref{lem:local_model_properties}} \\
        &= i\left( (L_{\mu_{_V}(v)}C)^{0}\right) & \text{by \cref{lem:representationdelzant}} \\
        &= i\left( (L_{i^*(\xi)}C)^{0}\right) & \text{because $\mu_{_V} (v) = i^* (\xi)$} \\ 
        &= L_\xi\left((i^*)^{-1}(C) \right)^0 & \text{by \cref{cor:tangenttopreimage}} \\
        &= \left( L_\xi P\right)^0 & \text{by \cref{lem:local_model_properties}.}
    \end{align*}
    
    \item[2.] Since $\mu_{_E}^{-1} (\xi)$ is a homogeneous space modeled by
    $G / G_{[g,\xi,v]}$, it is a torus.  By the above computation, its dimension is
    \[
    \dim \fg - \dim \fg_{[g,\xi,v]} = \dim L_\xi P. \qedhere
    \]
\end{enumerate}
\end{proof}


\subsection{Symplectic Slice Theorem}
\label{sec:slice_thm}
Let $G$ be a torus and $(M,\omega,\mu)$ a Hamiltonian $G$-space, i.e.,
a symplectic manifold $(M,\omega)$ with a Hamiltonian $G$-action for which
$\mu$ is a moment map.
Let $G_\ker \subset G$ be the action kernel, i.e., the subgroup of all
torus elements acting trivially.

For any $p\in M$, the orbit $G\cdot p$ is an isotropic submanifold, and the vector space
\[
V_p=\frac{T_p(G\cdot p)^\omega}{T_p(G\cdot p)},
\]
called \textit{symplectic slice representation} at $p$, inherits a symplectic structure from $(T_pM,\omega_p)$, denoted $\omega_{p,\red}$.
It is equipped with a linear action of the stabilizer $G_p$
obtained by differentiating the $G$-action on $M$. 

\begin{proposition}[Symplectic Slice Theorem \cite{meinrenkensymplectic}, \cite{ortega2001symplecticslicetheorem}]
\label{prop:slicetheorem}
    Let $(M,\omega,\mu)$ be a Hamiltonian $G$-space, and let $E$ be the symplectic local model (\cref{def:localmodel})
    defined by the symplectic slice representation $V_p$ at $p\in M$.
    Then there exists a $G$-equivariant symplectomorphism $\varphi:U_M \to U_E$ between neighborhoods of $G \cdot p\subseteq M$ and $G \cdot [e,0,0]\subseteq E$  such that \begin{enumerate}
        \item $\varphi:p\mapsto [e,0,0]$ and
        \item $\mu -\mu(p)= \mu_E \circ \varphi$.
    \end{enumerate} 
\end{proposition}

We next compute the kernel of the symplectic slice representation.
Recall:

\begin{lemma}
\label{lem:trivialactingelement}
Let $M$ be a connected $G$-manifold.
If $g\in G$ acts trivially on a non-empty open subset $U\subseteq M$,
then $g$ acts trivially on $M$.
\end{lemma}

\begin{proof}
    Let $H=\overline{\expval{g}}$ be the closed subgroup of $G$ generated by $g$.
    By construction, $H$ acts trivially on $U$.
    The connected components of the fixed point set $M^H$ are closed submanifolds of $M$,
    see e.g. \cite[Corollary B.39]{guillemin2002moment}.
    Reducing to a connected component if necessary, we may assume that $U$ is connected.
    The connected component of $M^H$ containing $U$ is open in $M$ since it is of full dimension.
    Hence, this component is both open and closed; therefore, it equals all of $M$, as the latter is connected.
\end{proof}

\begin{lemma}
\label{lem:kernel}
    Let $(M,\omega,\mu)$ be a connected Hamiltonian $G$-manifold, and let $\rho:G_p \to \sp(V_p,\omega_{p,\red})$ be the symplectic slice representation at $p\in M$. Then 
    \begin{equation*}
        \opker(\rho) = G_\ker.
    \end{equation*}
\end{lemma}

\begin{proof}
    The inclusion $G_\ker \subseteq \opker(\rho)$ is clear.
    For the converse inclusion, consider a $h\in \opker(\rho)$ and
    identify a neighborhood of $p$ with the local symplectic model $E$
    via the Symplectic Slice Theorem. By \cref{lem:local_model_properties}, $h$ fixes this open neighborhood of $p$.
    By \cref{lem:trivialactingelement}, $h$ acts trivially on $M$; hence $h\in G_\ker$. 
\end{proof}


\subsection{Toric Actions with Nontrivial Kernels}
\label{sec:nontrivial_kernels}

Let $G$ be a torus.

\begin{definition}
A \textit{symplectic toric $G$-manifold} is a triple $(M,\omega, \mu)$,
where $(M,\omega)$ is a compact connected symplectic manifold 
equipped with a symplectic $G$-action and a moment map 
$\mu : M \to \fg^*$, such that the kernel $G_\ker$ of the action
is connected and $\dim( G )- \dim( G_\ker) = \frac 12 \dim( M)$.\footnote{For consistency, we also view the empty set as a symplectic toric $G$-manifold.}
\end{definition}

We use the results from \cref{sec:slice_thm} to show that the symplectic slice representations $V_p$
are symplectic toric $G_p$-representations (\cref{def:s_toric_rep})
when $(M,\omega,\mu)$ is a toric $G$-manifold.

\begin{corollary} \label{cor:slicerepsaretoric}
    Let $(M,\omega,\mu)$ be a symplectic toric $G$-manifold. The stabilizer subgroup $G_p\subseteq G$ of any point $p\in M$ is a subtorus, and the symplectic slice representation $V_p$ at $p$ is a symplectic toric $G_p$-representation.
\end{corollary}

\begin{proof}
    Any stabilizer subgroup is closed, so it suffices to show that $G_p$ is connected. It follows from \cref{lem:kernel} that the induced representation of $G_p/G_\ker$ on $V_p$ is faithful. We can thus interpret this quotient as a compact abelian subgroup of the unitary group $U(V_p)$. Since $U(V_p)$ has rank \begin{align*}
        \frac{1}{2} \dim(V_p) = \dim\left(\frac{G_p}{G_\ker}\right),
    \end{align*}
    the connected component $(G_p/G_\ker)_0$ is a maximal torus in $U(V_p)$. Since maximal tori are maximal amongst abelian subgroups \cite[Theorem 11.36]{hall2013lie}, we deduce that $G_p/G_\ker=(G_p/G_\ker)_0$. Since both $G_\ker$ and $G_p/G_\ker$ are connected, we conclude that $G_p$ is connected.
    By \cref{lem:kernel}, the kernel of the $G_p$-representation on $V_p$
    is the subtorus $G_\ker$, which has dimension $\dim(G_p) - \frac12\dim(V_p)$.
\end{proof}

\begin{corollary}[Delzant's Lemma for Manifolds]
\label{cor:delzants_lemma}
    Let $(M,\omega,\mu)$ be a symplectic toric $G$-manifold with moment polytope $\Delta\subseteq \fg^*$. Then
    \begin{enumerate}
        \item $\mu:M\to \fg^*$ is the point-orbit map;
        \item for $\xi\in \Delta$, the fiber $\mu^{-1}(\xi)$ is a torus with $\dim(\mu^{-1}(\xi))=\dim(L_\xi \Delta)$;
        \item for $p\in M$, the stabilizer $G_p\subseteq G$ is the subtorus with $\fg_p^0 = L_{\mu(p)}\Delta$.
    \end{enumerate}
\end{corollary}

\begin{proof}
    By \cref{cor:slicerepsaretoric}, the local symplectic models of $M$ are toric. The result then follows immediately from \cref{prop:slicetheorem} and \cref{lem:delzantlocalmodel}.
\end{proof}


\subsection{Classification by Moment Polytopes}
\label{sec:classification}

Let $(M,\omega, \mu)$ be a symplectic toric $G$-manifold
with action kernel $G_\ker$.
By the Convexity Theorem~\cite{atiyah1982convexity,guillemin1982convexity},
the image of the moment map $\mu$,
\[
   \Delta := \mu (M)
\]
is the convex hull of the images of the fixed points of the action.
This is called the \textit{moment polytope}. Applying the symplectic slice theorem (\cref{prop:slicetheorem}) and \cref{prop:geometricstructure} at any fixed point shows that the affine subspace spanned by the moment polytope $\Delta$ is a translate of
$\fg_\ker^0 \subseteq \fg^*$ and that $\Delta$ is \emph{unimodular} (\cref{def:unimodular_polytope}).

We next turn to the classification of symplectic toric $G$-manifolds
modulo \textit{isomorphism}:

\begin{definition}
Symplectic toric $G$-manifolds, $(M_1,\omega_1, \mu_1)$,
and $(M_2,\omega_2, \mu_2)$ are \textit{isomorphic} 
if there exists an equivariant symplectomorphism intertwining
the moment maps; that is, a symplectomorphism
$\varphi : (M_1,\omega_1) \to (M_2,\omega_2)$ such that
$\varphi (g \cdot p) = g \cdot \varphi (p)$ and
$\mu_2 (\varphi (p)) = \mu_1 (p)$ for all $g \in G$ and $p \in M_1$.
\end{definition}



In the case $G_\ker=\{1\}$, Delzant~\cite{delzant1988hamiltoniens} showed that
unimodular polytopes in $\fg^*$
classify symplectic toric $G$-manifolds up to isomorphism,
where the correspondence is given by the moment map image,
$\Delta = \mu (M)$.
The case of arbitrary symplectic toric $G$-manifolds -- where $G_\ker$ might be nontrivial --
can be reduced to this special case by considering the induced action of the quotient torus $G/G_\ker$, which is effective by construction.
We thus have the following slight extension of Delzant's Theorem,
where unimodularity is in the extended sense of \cref{def:unimodular_polytope}:

\begin{proposition}[Delzant's Theorem with Nontrivial Action Kernels]
\label{thm:delzant}
Symplectic toric $G$-manifolds are classified, up to isomorphism,
by unimodular polytopes in $\fg^*$, where the symplectic toric manifold
$(M,\omega,\mu)$ is associated with the moment polytope $\mu(M)$.
\end{proposition}


\subsection{Symplectic Toric Submanifolds}
\label{sec:submanifolds}

Let $(M,\omega, \mu)$ be a symplectic toric $G$-manifold.

\begin{definition}
A \textit{symplectic toric submanifold} of $(M,\omega, \mu)$
is a $G$-invariant closed connected symplectic submanifold of $(M,\omega)$
equipped with the restriction of the $G$-action and of $\mu$.
\end{definition}


Since the moment polytope $\mu(M)$ is the $G$-orbit space of
$(M,\omega, \mu)$, it is natural to analyze the image under the
moment map of its symplectic toric submanifolds.

\begin{proposition}
\label{prop:submanifolds_are_faces}
Let $(M,\omega,\mu)$ be a symplectic toric $G$-manifold. Then:
\begin{itemize}
\item
The symplectic toric submanifolds of $(M,\omega, \mu)$
are the preimages of the faces of its moment polytope, $\mu(M)$;
\item
Any symplectic toric submanifold of $(M,\omega, \mu)$
is itself a symplectic toric $G$-manifold.
\end{itemize}
\end{proposition}

\begin{proof}
\begin{itemize}
    \item Let $(N,\iota^* \omega, \iota^* \mu)$, shortly denoted as $N$,
be a symplectic toric submanifold
of $(M,\omega, \mu)$, where $\iota : N \hookrightarrow M$ is the inclusion map.

We first show that the image of $\iota^* \mu : N \to \fg^*$ is a face
of the moment polytope $\Delta_M := \mu(M)$.
Since $N$ is a hamiltonian $G$-space, we have by the Convexity
Theorem~\cite{atiyah1982convexity,guillemin1982convexity} that
$\Delta_N := \iota^*\mu (N)$ is the convex hull of the images of the
fixed points contained in $N$.
Hence $\Delta_N$ is the convex hull of a subset of the vertices of $\Delta_M$.

Consider the symplectic slice representation $T_pM$ at any fixed point $p\in N$,
and note that $T_pN$ is a symplectic toric subrepresentation.
By Proposition~\ref{prop:representationdelzant},
the local cone $C_\xi(\Delta_N)$ is a face of the local cone
$C_\xi(\Delta_M)$.
But the convex hull of a subset of vertices with this property must be a face (see Lemma~\ref{lem:coneface_to_polytopeface}).

Conversely, let $F$ be a face of $\Delta_M$
and consider $N:= \mu^{-1}(F)$.
Since $\mu$ is continuous and a point-orbit map,
$N$ is closed and $G$-invariant.
Since $F$ is a face of the rational polytope $\Delta_M$, there exists a rational supporting hyperplane $H_{(v,c)} = \{\xi\in \fg^*\mid \expval{\xi,v}=c \}$
such that $F= \Delta_M \cap H_{(v,c)}$.
Since $H_{(v,c)}$ is rational, $v\in \fgz$ determines an $S^1$-subgroup of $G$ (see \cref{ex:integrallattice}) yielding
a Hamiltonian $S^1$-action on $M$.
The corresponding moment map $\mu_{S^1}$ is given by restriction of each
$\mu(p)\in \fg^*$ to $\fs := {\rm{span}} \{ v \}$.
Therefore,
\begin{equation*}
   p\in N \quad \Longleftrightarrow \quad \mu(p) \in F
   \quad \Longleftrightarrow \quad \expval{\mu(p),v}=c
\end{equation*}
and $N$ is a level set of $\mu_{S^1}$.
By the Convexity Theorem, it is connected. 

Furthermore, $\mu_{S^1}^v:p \mapsto \expval{\mu_{S^1}(p),v}$ is a Morse-Bott function and the connected components of $\text{Crit}\left(\mu_{S^1}^v\right)$ are symplectic submanifolds.
By construction, $N$ is such a critical submanifold and thus a symplectic toric submanifold.

\item Finally, we show that $(N,\iota^*\omega,\iota^*\mu)$ is itself a symplectic toric manifold.
It only remains to check the conditions on the kernel, $G_N$, of the $G$-action on $N$.
By \cref{cor:delzants_lemma}, $G_N$ is the subtorus of $G$ whose Lie algebra is the
annihilator of the linear space of $F$,
hence $G_N$ is connected.
The dimension condition follows from the local version for representations
in Proposition \ref{prop:representationdelzant} by considering the symplectic slice
representation $T_pM$ and its subrepresentation $T_pN$,
for an arbitrary fixed point $p\in N$. \qedhere
\end{itemize}
\end{proof}


\section{Reduction and Integral Transversality}
\label{sec:reduction_1}

This section is where the machinery of \cref{sec:tori_framework,sec:stm_framework}
starts to pay off:
we characterize reduction directly in terms of the moment polytope.
We first show
that regular levels of the moment map correspond to transverse affine
subspaces of the polytope (\cref{sec:reg_levels})
and point to the criteria that make this
transversality easy to check in practice (\cref{cor:criteriaforaffinetransversality}).
We then upgrade this to integral transversality, obtaining a numerical
criterion for exactly when a reduction level yields a smooth reduced
space (\cref{sec:reduction_levels}).
The section closes with polytope intersection
(\cref{sec:intersection}), realized via reduction of a product manifold
by the skew-diagonal torus, which extends the constructions above to
two polytopes at once and, in the noncompact case, recovers cutting
itself as a special case.


\subsection{Regular Levels and Relative Transversality}
\label{sec:reg_levels}

Let $(M,\omega, \mu)$ be a symplectic toric $G$-manifold,
with action kernel $G_\ker$. Let $H$ be a subtorus of $G$,
and $\fh^0 \subseteq \fg^*$ the
annihilator of the Lie algebra of $H$. We will write $H_\ker=H\cap G_\ker$ for the kernel of the restricted $H$-action.

\begin{definition}
A translate $A$ of $\fh^0$ is an $H$-\textit{regular level}
for $(M,\omega, \mu)$ 
if $H/H_\ker$ acts locally freely on $\mu^{-1}(A)$.
\end{definition}

\begin{remark}
    This terminology is justified by the fact that
    a translate $A$ of $\fh^0$ is a level of the moment map
    $i^* \circ \mu : M \to \fh^*$
    for the action restricted to $H$; here $i:\fh \hookrightarrow \fg$ is the inclusion.
    By the properties of the differential of the moment map, that level is regular
    exactly when the action of $H/H_\ker$ is locally free there.
    In this case, $\mu^{-1}(A)$ is a smooth coisotropic submanifold of $M$.
\end{remark}

We want to characterize regular levels using the moment polytope.

\begin{definition}
\label{def:transv}
    Let $\Delta\subseteq \fg^*$ be a polytope,
    and let $A\subseteq \fg^*$ be an affine subspace with direction subspace $B$
    (i.e., $A$ is a translate of the subspace $B$).
    We say that $A$ intersects $\Delta$ \textit{relatively transversely} at $\xi \in A\cap \Delta$ if
    \[
    B + L_\xi \Delta =  B + \dir{\Delta}.
    \]
    We say that $A$ and $\Delta$ intersect \emph{relatively transversely}, denoted $A \pitchfork \Delta$, if they intersect transversely at all their intersection points.
\end{definition}

\begin{remark}
\label{rmk:usual_transv_1}
If the polytope has full dimension, i.e., $\dir{\Delta}=\fg^*$, relative transversality
from \cref{def:transv} reduces to the familiar notion of transversality:
        \begin{equation*}
             B + L_\xi \Delta = \fg^* \qquad \text{for all } \xi\in A \cap \Delta.
        \end{equation*}
        More generally, intersecting both sides of the relative transversality condition with $\dir{\Delta}$ and applying the modular law for vector spaces with $L_\xi\Delta \subseteq \dir{\Delta}$ yields
\begin{equation*}
    (B\cap \dir{\Delta}) + L_\xi \Delta = \dir{\Delta} \qquad \text{for all } \xi\in A \cap \Delta. 
\end{equation*}
        \end{remark}

In \cref{sec:polyhedral_cones_and_polytopes}, we state and prove practical criteria for and consequences of relative transversality. In particular, we show that 
\begin{itemize}
    \item $A\pitchfork \Delta$ precisely if $A$ does not intersect any faces of $\Delta$ whose dimension is lower that $k=\dim(B+\dir{\Delta})-\dim(B)$ (\cref{cor:criteriaforaffinetransversality}) and that
    \item if $A\pitchfork\Delta$, the vertices of $A\cap \Delta$ are the points where $A$ intersects a $k$-dimensional face of $\Delta$ (\cref{cor:consequencesofaffinetransversality}).
\end{itemize} 
The main point of this section is that a rational affine subspace intersects the moment polytopes relatively transversely precisely if it is a regular level:
\begin{proposition}[Regular Levels in Terms of Transversality]
\label{prop:regularlevel}
    Let $(M,\omega,\mu)$ be a symplectic toric $G$-manifold, let $H$ be a subtorus of $G$, and let $\fh$ be the Lie algebra of $H$. A translate $A$ of $\fh^0$ is a regular $H$-level if and only if it intersects $\Delta=\mu(M)$ transversely.
\end{proposition}

\begin{proof}
    The condition of $H/H_\ker$ acting locally freely is equivalent to all stabilizers in $H$ having Lie algebra $\fh_\ker=\fh\cap \fg_\ker$. This means that for any $\xi\in A \cap  \Delta$ and any $p\in \mu^{-1}(\xi)$, we have 
    \begin{equation*}
        \fh \cap \fg_p = \fh \cap \fg_\ker.
    \end{equation*}
    Taking the annihilator, this is equivalent to 
    \begin{equation*}
        \fh^0 + \fg_p^0 = \fh^0 + \fg_\ker^0.
    \end{equation*}
    Using \cref{cor:delzants_lemma}, we identify
    $L_\xi \Delta=\fg_p^0$ and $\dir{\Delta} = \fg_\ker^0$ so that the above condition becomes
    \begin{equation*}
        B + L_\xi \Delta = B+\dir{\Delta},
    \end{equation*}
    which is exactly the relative transversality condition at $\xi \in A\cap \Delta$.
\end{proof}

We will now repeat the above argument with integral transversality (instead of relative transversality) and annihilators in the sense of \cref{thm:annihilatorfortorus} (instead of vector space annihilators) to characterize levels on which $H/H_\ker$ acts freely.


\subsection{Reduction Levels and Integral Transversality}
\label{sec:reduction_levels}

Let $(M,\omega, \mu)$ be a symplectic toric $G$-manifold,
with kernel $G_\ker$.
Let $H$ be a subtorus of $G$, let $\fh^0 \subseteq \fg^*$ be the
annihilator of the Lie algebra $\Lie (H) = \fh$,
and let $H_\ker=H\cap G_\ker$ be the kernel of the restricted $H$-action.

\begin{definition}
A translate $A$ of $\fh^0$ is an $H$-\textit{reduction level}
for $(M,\omega, \mu)$ 
if $H/H_\ker$ acts freely on $\mu^{-1} (A)$.
\end{definition}

We want to characterize reduction levels using the moment polytope.
For this, we generalize the notion of integral transversality
due to McDuff \cite{mcduff2011displacing},
which was already encountered \cref{rmk:integral_transv_origin}:

\begin{definition}
\label{def:int_transv}
    Let $\Delta \subseteq \fg^*$ be a rational polytope,
    and let $A\subseteq \fg^*$ be a translate of a rational subspace $B$.
    We say that $A$ intersects $\Delta$ \textit{integrally transversely} at $\xi \in A\cap \Delta$ if
    \begin{equation*}
         (B \cap \fgz^*) + (L_\xi \Delta \cap \fgz^*) =
         (B \cap \fgz^*) + (\dir{\Delta} \cap \fgz^*).
    \end{equation*}
    We say that $A$ and $\Delta$ intersect \emph{integrally transversely}, denoted $A \pitchfork_{_\ZZ} \Delta$, if they intersect integrally transversely at all their intersection points. 
\end{definition}

\begin{remark}
    Integral transversality implies relative transversality (\cref{def:transv}), since these
    vector spaces are spanned by the corresponding lattices.
\end{remark}

In \cref{sec:polyhedral_cones_and_polytopes}, we show that it is enough to check the integral transversality condition at the vertices of $A \cap \Delta$ (\cref{lem:criteriaforintegraltransversality}). Together with the criteria for relative transversality, this provides a clear recipe on how to check whether a rational affine subspace $A$ intersects a rational polytope $\Delta$ integrally transversely in practice:
\begin{enumerate}
    \item Determine $k=\dim(B+\dir{\Delta}) - \dim(B) = \dim(\Delta) - \dim(B\cap \Delta)$.
    \item Verify that $A$ does not intersect any faces whose dimension is less than $k$.
    \item Determine the vertices of $A\cap \Delta$ by finding the points where $A$ intersects a face of dimension $k$.
    \item Verify the integral transversality condition at the vertices of $A\cap \Delta$.
\end{enumerate}

\begin{example}
\label{ex:transv_2d}
These examples refer to the planar images in \cref{fig:transv_2d}:
\begin{itemize}
    \item 
The rectangle on the left intersects the line of slope 1
integrally transversely.

\item
The intersection of the rectangle in the middle with the line of slope 1 fails transversality at the lower left vertex, marked with a bullet.

\item
The triangle on the right intersects the line of slope 1 transversely, but the integral transversality fails at the point $\xi$,
since $L_\xi \Delta \cap \fgz^* = \spn{\ZZ}{\binom{1}{-1}}$,
$L_\xi A \cap \fgz^* = \spn{\ZZ}{\binom{1}{1}}$,
but these do not sum up to $\fgz^*$.
\end{itemize}
In all three cases, the intersection is a \textit{bona fide}
unimodular polytope, namely a line segment of slope 1.
\end{example}

\begin{figure}[ht]
\begin{tikzpicture}[scale=0.66]
  \tikzstyle{every path}=[thick]
  \tikzstyle{pt}=[circle, fill=black, inner sep=2pt]

    \filldraw[fill=blue!20,fill opacity=0.25] (0,0) -- (2,0) -- (2,3) -- (0,3) -- cycle;
    \draw (-2,-1.5) -- (4,4.5);

  \begin{scope}[xshift=7cm]
    \filldraw[fill=blue!20,fill opacity=0.25] (0,0) -- (2,0) -- (2,3) -- (0,3) -- cycle;
    \draw (-1.5,-1.5) -- (4.5,4.5);
    \fill (0,0) circle (3pt);
  \end{scope}

  \begin{scope}[xshift=14cm]
    \filldraw[fill=blue!20,fill opacity=0.25] (0,0) -- (3,0) -- (0,3) -- cycle;
    \draw (-2,-1.5) -- (4,4.5);
    \fill (1.25,1.75) circle (3pt) node[scale=1,above] at (1.25,1.75) {$\xi$};
  \end{scope}

\end{tikzpicture}
\caption{Illustrations of relative and integral transversality for \cref{ex:transv_2d}.}
\label{fig:transv_2d}
\end{figure}
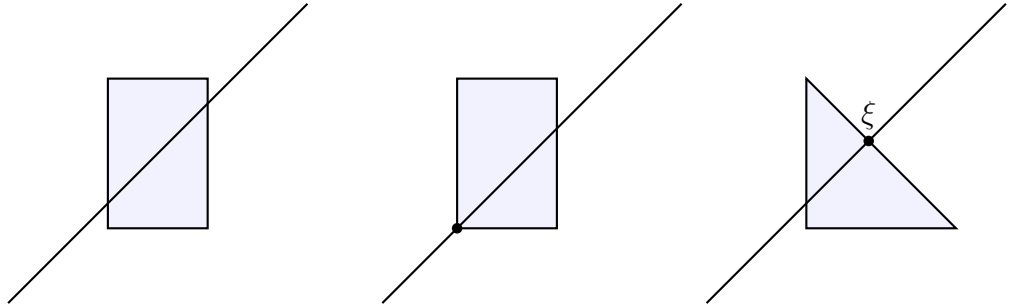

\begin{example}
\label{ex:transv_3d}
These examples refer to the images in \cref{fig:transv_3d}.
\begin{itemize}
    \item 
Image on the left:\\
The rational plane $x+2z=4$ intersects the tetrahedron with vertices
$(0,0,0), (3,0,0), (0,3,0), (0,0,3)$ transversely in a triangle, but not integrally transversely.
Integral transversality fails, for instance, at $\xi := (0,1,2)$, where
\[
L_\xi \Delta \cap \fgz^* =
\spn{\ZZ}{\textstyle{\begin{matrix} 0 \\ 1 \\ -1 \end{matrix}}}
\quad \text{ and } \quad
L_\xi A \cap \fgz^* =
\spn{\ZZ}{\textstyle{\begin{pmatrix} 2 \\ 0 \\ -1 \end{pmatrix},\begin{pmatrix} 0 \\ 1 \\ 0 \end{pmatrix}}}
\]
do not sum up to $\fgz^*$.

\item 
Second image from the left:\\
The rational (horizontal) plane $z=1$ intersects the same tetrahedron
integrally transversely in an isosceles triangle.

\item 
Second image from the right:\\
The rational plane $x+y+2z=4$ intersects the same tetrahedron
transversely, but not integrally transversely, also in an isosceles triangle.

\item
Image on the right:\\
The rational plane $x+2z=3$ intersects the same tetrahedron not even transversely:
transversality fails at the intersection vertex $(3,0,0)$.

\end{itemize}
\end{example}

\begin{figure}[ht]
\centering
\begin{tikzpicture}[tdplot_main_coords, scale=0.7]

  \begin{scope}[xshift=0cm]

  \coordinate (O) at (0,0,0);
  \coordinate (A) at (3,0,0); 
  \coordinate (B) at (0,3,0); 
  \coordinate (C) at (0,0,3); 

  \coordinate (D) at (0,1,2);
  \coordinate (E) at (2,0,1);
  \coordinate (F) at (0,0,2);

  \draw[thick] (O) -- (A);
  \draw[thick] (O) -- (B);
  \draw[thick] (O) -- (C);
  \draw[thick] (A) -- (B);
  \draw[thick] (B) -- (C);
  \draw[thick] (C) -- (A);

  \draw[thick] (D) -- (E) -- (F) -- cycle;

  \filldraw[fill=blue!10, opacity=0.5] (O) -- (A) -- (B) -- cycle;
  \filldraw[fill=blue!10, opacity=0.5] (O) -- (A) -- (C) -- cycle;
  \filldraw[fill=blue!10, opacity=0.5] (O) -- (B) -- (C) -- cycle;
  \filldraw[fill=blue!10, opacity=0.5] (A) -- (B) -- (C) -- cycle;

  \filldraw[fill=blue!30, opacity=0.5] (D) -- (E) -- (F) -- cycle;

  \fill (D) circle (2pt) node[scale=1,above right] at (D) {$\xi$};
  \fill (F) circle (2pt);
  \end{scope}

  \begin{scope}[xshift=5cm]

  \coordinate (O) at (0,0,0);
  \coordinate (A) at (3,0,0); 
  \coordinate (B) at (0,3,0); 
  \coordinate (C) at (0,0,3); 

  \coordinate (D) at (0,2,1);
  \coordinate (E) at (2,0,1);
  \coordinate (F) at (0,0,1);

  \draw[thick] (O) -- (A);
  \draw[thick] (O) -- (B);
  \draw[thick] (O) -- (C);
  \draw[thick] (A) -- (B);
  \draw[thick] (B) -- (C);
  \draw[thick] (C) -- (A);

  \draw[thick] (D) -- (E) -- (F) -- cycle;

  \filldraw[fill=blue!10, opacity=0.5] (O) -- (A) -- (B) -- cycle;
  \filldraw[fill=blue!10, opacity=0.5] (O) -- (A) -- (C) -- cycle;
  \filldraw[fill=blue!10, opacity=0.5] (O) -- (B) -- (C) -- cycle;
  \filldraw[fill=blue!10, opacity=0.5] (A) -- (B) -- (C) -- cycle;

  \filldraw[fill=blue!30, opacity=0.5] (D) -- (E) -- (F) -- cycle;
  \end{scope}

  \begin{scope}[xshift=10cm]

  \coordinate (O) at (0,0,0);
  \coordinate (A) at (3,0,0); 
  \coordinate (B) at (0,3,0); 
  \coordinate (C) at (0,0,3); 

  \coordinate (D) at (0,2,1);
  \coordinate (E) at (2,0,1);
  \coordinate (F) at (0,0,2);

  \draw[thick] (O) -- (A);
  \draw[thick] (O) -- (B);
  \draw[thick] (O) -- (C);
  \draw[thick] (A) -- (B);
  \draw[thick] (B) -- (C);
  \draw[thick] (C) -- (A);

  \draw[thick] (D) -- (E) -- (F) -- cycle;

  \filldraw[fill=blue!10, opacity=0.5] (O) -- (A) -- (B) -- cycle;
  \filldraw[fill=blue!10, opacity=0.5] (O) -- (A) -- (C) -- cycle;
  \filldraw[fill=blue!10, opacity=0.5] (O) -- (B) -- (C) -- cycle;
  \filldraw[fill=blue!10, opacity=0.5] (A) -- (B) -- (C) -- cycle;

  \filldraw[fill=blue!30, opacity=0.5] (D) -- (E) -- (F) -- cycle;

  \fill (D) circle (2pt);
  \fill (E) circle (2pt);
  \fill (F) circle (2pt);
\end{scope}

  \begin{scope}[xshift=15cm]

  \coordinate (O) at (0,0,0);
  \coordinate (A) at (3,0,0); 
  \coordinate (B) at (0,3,0); 
  \coordinate (C) at (0,0,3); 

  \coordinate (D) at (0,2,1);
  \coordinate (E) at (3,0,0);
  \coordinate (F) at (0,0,1);

  \draw[thick] (O) -- (A);
  \draw[thick] (O) -- (B);
  \draw[thick] (O) -- (C);
  \draw[thick] (A) -- (B);
  \draw[thick] (B) -- (C);
  \draw[thick] (C) -- (A);

  \draw[thick] (D) -- (E) -- (F) -- cycle;

  \filldraw[fill=blue!10, opacity=0.5] (O) -- (A) -- (B) -- cycle;
  \filldraw[fill=blue!10, opacity=0.5] (O) -- (A) -- (C) -- cycle;
  \filldraw[fill=blue!10, opacity=0.5] (O) -- (B) -- (C) -- cycle;
  \filldraw[fill=blue!10, opacity=0.5] (A) -- (B) -- (C) -- cycle;

  \filldraw[fill=blue!30, opacity=0.5] (D) -- (E) -- (F) -- cycle;

  \fill (D) circle (2pt);
  \fill (E) circle (2pt);
  \fill (F) circle (2pt);
  \end{scope}

\end{tikzpicture}
\caption{Illustrations regarding transversality and integral transversality in 3-dimensional space described in \cref{ex:transv_3d}.
The vertices of the intersection where integral transversality fails are marked with bullets.}
\label{fig:transv_3d}
\end{figure}
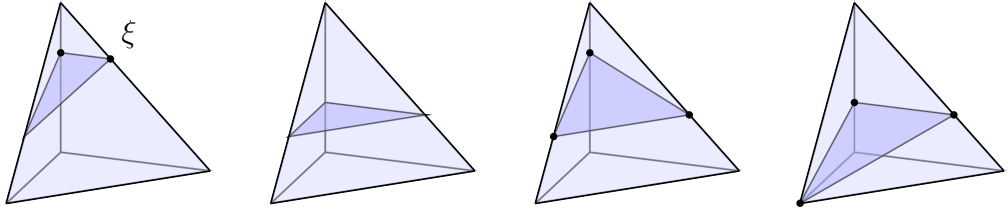

We are now ready to state and prove the main result of this section:
\begin{proposition}[Reduction Levels in Terms of Integral Transversality]
\label{prop:reduction_levels}
    Let $(M,\omega,\mu)$ be a symplectic toric $G$-manifold, and let $H\subseteq G$ be a subtorus with Lie algebra $\fh$.
    A translate $A$ of $\fh^0$ is an $H$-reduction level if and only if it intersects $\Delta := \mu(M)$ integrally transversely.
\end{proposition}

\begin{proof}
    The condition of $H/H_\ker$ acting freely is equivalent to all stabilizers in $H$ being equal to $H_\ker$. In other words, for any $\xi\in A \cap  \Delta$ and any $p\in \mu^{-1}(\xi)$, we must have
    \begin{equation*}
        H\cap G_p = H \cap G_\ker.
    \end{equation*}
    Applying the annihilator mechanism (\cref{cor:subtorusintersection}), this is equivalent to 
    \begin{equation*}
        (\fh^0\cap \fgz^*) + (\fg_p^0\cap \fgz^*) = (\fh^0 \cap \fgz^*) + (\fg_\ker^0 \cap \fgz^*).
    \end{equation*}
    By Delzant's Lemma (\cref{cor:delzants_lemma}), we identify $\fg_p^0= L_\xi \Delta$ and $\fg_\ker^0 = \dir{\Delta}$ so that this becomes
    \begin{equation*}
        (B\cap \fgz^*) + (L_\xi \Delta\cap \fgz^*) = (B \cap \fgz^*) + (\dir{\Delta}\cap \fgz^*). \qedhere
    \end{equation*}
\end{proof}

Toric reduction is reformulated as follows in terms of integral transversality.

\begin{corollary}
\label{cor:reduced_space}
Let $(M,\omega,\mu)$ be a symplectic toric $G$-manifold,
and let $H\subseteq G$ be a subtorus.
If a translate $A\subseteq \fg^*$ of $\fh^0$ intersects
$\Delta := \mu(M)$ integrally transversely,
the quotient $\mu^{-1} (A)/H$ inherits the structure of a symplectic toric $G$-manifold $(M_\red,\omega_\red,\mu_\red)$, where the symplectic form is determined by $pr^*\omega_\red =\iota^*\omega$ and the moment map by the commutativity of the diagram 
\begin{equation*}
    \begin{tikzcd}
\mu^{-1}(A) \arrow[r, "\iota", hook] \arrow[d, "pr"',"/H", two heads] & {(M,\omega)} \arrow[d, "\mu"] \\
{(M_\red,\omega_\red)} \arrow[r, "\mu_\red"]                & \fg^*.
\end{tikzcd}
\end{equation*}
\end{corollary}

$(M_\red,\omega_\red,\mu_\red)$ is called the \textit{reduced space} at the level $A$.
Its moment polytope is precisely $A \cap \Delta$, which is unimodular.
The kernel of the $G$-action on $M_\red$ is $HG_\ker$, where $G_\ker$ is the kernel of the $G$-action on $M$.

The following corollary is usually viewed as a consequence of
Delzant's classification of symplectic toric manifolds~\cite{delzant1988hamiltoniens}.
In \cref{sec:polyehedronintersection}, we give an alternative proof in terms of integral transversality and unimodularity.

\begin{corollary}
\label{cor:int_transv_gives_unimodular}
    Let $\Delta \subseteq \fg^*$ be a unimodular polytope, and let $A\subseteq \fg^*$ be a rational affine subspace. If $A \pitchfork_{_\ZZ} \Delta$,
    then the intersection $A \cap \Delta$ is unimodular.
\end{corollary}

Also in \cref{sec:polyehedronintersection}, we show the following result, which, in view of \cref{prop:reduction_levels}, is the analogue of reduction in stages:

\begin{corollary}[Recursive Aspects of Reduction]
    Let $\Delta \subseteq \fg^*$ be a unimodular polytope,
    $A' \subseteq A \subseteq \fg^*$ rational affine subspaces,
    and $A \pitchfork_{_\ZZ} \Delta$. 
    Then
    \[
    A' \pitchfork_{_\ZZ} \Delta
    \; \iff \;
    A'\pitchfork_{_\ZZ} (A \cap \Delta).
    \]
\end{corollary}


\subsection{Polytope Intersection}
\label{sec:intersection}

We extend the notions of relative and integral transversality
to intersections of (rational) polytopes.

\begin{definition}
\label{def:integral_transv_polytopes}
Let $\Delta_1$ and $\Delta_2$ be two polytopes in $\fg^*$. \begin{itemize}
    \item We say that they intersect \emph{relatively transversely} at $\xi\in \Delta_1\cap \Delta _2$ when
    \begin{equation*}
        L_\xi \Delta_1 + L_\xi\Delta_2 = \dir{\Delta_1} + \dir{\Delta_2}.
    \end{equation*}
    The polytopes intersect \textit{relatively transversely}, denoted as $\Delta_1 \pitchfork \Delta_2$, when they intersect  transversely at all their intersection points.
    \item Assuming that $\Delta_1$ and $\Delta_2$ are rational, we say that they \textit{intersect integrally transversely at $\xi \in \Delta_1 \cap \Delta_2$} when
    \[
        (L_\xi \Delta_1 \cap \fgz^*) + (L_\xi \Delta_2\cap \fgz^*) = (\dir{\Delta_1}\cap \fgz^*) + (\dir{\Delta_2}\cap \fgz^*).
    \]
    The polytopes \textit{intersect integrally transversely}, denoted as  $\Delta_1 \pitchfork_{_\ZZ} \Delta_2$, when they intersect integrally transversely at all their intersection points.
\end{itemize}
\end{definition}

Just as for the intersection of a polytope with an affine subspace, there is a clear procedure to check whether two rational polytopes $\Delta_1$ and $\Delta_2$ intersect integrally transversely:
\begin{enumerate}
    \item Compute $d=\dim(\dir{\Delta_1}+\dir{\Delta_2})$. 
    \item $\Delta_1$ and $\Delta_2$ intersect relatively transversely if the dimensions of any two faces with a non-empty intersection sum to at least $d$ (\cref{lem:criteriafortransversality}). 
    \item The vertices of $\Delta_1\cap \Delta_2$ are the points where the dimensions of the faces sum exactly to $d$ (\cref{lem:consequencesoftransversality}).
    \item $\Delta_1$ and $\Delta_2$ intersect integrally transversely if the integral transversality condition holds at the vertices of $\Delta_1\cap \Delta_2$ (\cref{lem:criteriaforintegraltransversality}).
\end{enumerate}
   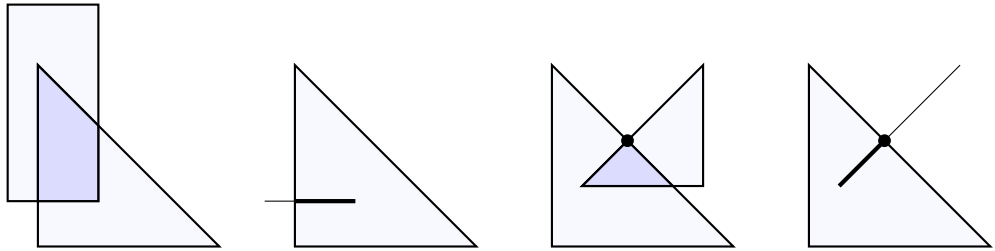
\begin{figure}[ht]
\begin{center}
\begin{tikzpicture}[scale=0.4]
  \tikzstyle{every path}=[thick]
  \tikzstyle{pt}=[circle, fill=black, inner sep=2pt]

    \filldraw[fill=blue!10,fill opacity=0.25] (0,1.5) -- (0,8) -- (3,8) -- (3,1.5) -- cycle;

    \filldraw[fill=blue!10,fill opacity=0.25] (1,0) -- (1,6) -- (7,0) -- cycle;

    \filldraw[fill=blue!40,fill opacity=0.25] (1,1.5) -- (1,6) -- (3,4) -- (3,1.5) -- cycle;

  \begin{scope}[xshift=8.5cm]
    \filldraw[fill=blue!10,fill opacity=0.25] (1,0) -- (1,6) -- (7,0) -- cycle;

    \draw[thin] (0,1.5) -- (1,1.5);
    \draw[ultra thick] (1,1.5) -- (3,1.5);
  \end{scope}

  \begin{scope}[xshift=17cm]
    \filldraw[fill=blue!10,fill opacity=0.25] (1,0) -- (1,6) -- (7,0) -- cycle;

    \filldraw[fill=blue!10,fill opacity=0.25] (2,2) -- (6,2) -- (6,6) -- cycle;

    \filldraw[fill=blue!40,fill opacity=0.25] (2,2) -- (5,2) -- (3.5,3.5) -- cycle;

    \fill (3.5,3.5) circle (6pt);
  \end{scope}

  \begin{scope}[xshift=25.5cm]
    \filldraw[fill=blue!10,fill opacity=0.25] (1,0) -- (1,6) -- (7,0) -- cycle;

    \draw[thin] (3.5,3.5) -- (6,6);
    \draw[ultra thick] (2,2) -- (3.5,3.5);
  
    \fill (3.5,3.5) circle (6pt);
  \end{scope}

\end{tikzpicture}
\caption{Illustrations of integral transverse (two figures on the left)
and only transverse (two figures on the right) polytope intersection.
The vertices where integral transversality fails are marked with bullets.}
\label{fig:intersection_2d}
\end{center}
\end{figure}

Analoguously to \cref{cor:int_transv_gives_unimodular}, we show the following result:
\begin{corollary}
Let $\Delta_1$ and $\Delta_2$ be unimodular polytopes in $\fg^*$.
If $\Delta_1 \pitchfork_{_\ZZ} \Delta_2$,
    then their intersection $\Delta_1 \cap \Delta_2$ is unimodular.
\end{corollary}
In particular, the integral transverse intersection of polytopes
$\Delta_1$ and $\Delta_2$ that are the moment polytopes of
symplectic toric $G$-manifolds
$(M_1,\omega_1, \mu_1)$ and $(M_2,\omega_2, \mu_2)$ is thus again the moment polytope of a symplectic toric $G$-manifold. We now exhibit this manifold as a reduction of the product manifold $M_1\times M_2$.

Consider the skew-diagonal inclusion
$i_\skew : G \hookrightarrow G \times G$, $t \mapsto (t,t^{-1})$,
inducing the linear immersion
$\fg  \hookrightarrow \fg \oplus \fg$, $X \mapsto (X,-X)$,
and the dual submersion
\[
   i_\skew^* : \fg^* \oplus \fg^* \twoheadrightarrow \fg^* \ ,
   \qquad (\xi_1 , \xi_2 ) \mapsto \xi_1 - \xi_2 \ .
\]

The restriction of the $(G \times G)$-action on the product
manifold $M_1 \times M_2$ to the skew-diagonal
subtorus $S := i_\skew (G)$ has moment map
$\pi_1^* \mu_1 - \pi_2^* \mu_2$,
where $\pi_i : M_1 \times M_2 \to M_i$ ($i=1,2$) are the factor projections, and kernel $S_\ker:= S\cap (G\times G)_\ker$.
The moment polytope for this $S$-action is given by
the Minkowski sum of $\Delta_1$ with $-\Delta_2$,\footnote{We avoid calling this
{\textit{Minkowski difference}}, since that has different meanings
in the literature. In particular, the sum with the symmetric polytope
is not the inverse operation of the sum with the original polytope.}
\[
   \Delta_1 + (-\Delta_2) :=
   \{ \xi_1 - \xi_2 \mid \xi_1 \in \Delta_1 \text{ and } \xi_2 \in \Delta_2 \}.
\]

\begin{proposition}[Polytope Intersection]
\label{prop:intersection}
For $i=1,2$, let $\Delta_i$ be the moment polytope of a symplectic toric manifold $(M_i,\omega_i, \mu_i)$. Let $S\subset G\times G$ be the skew-diagonal torus. If $\Delta_1$ and $\Delta_2$ intersect integrally transversely, the reduction of the symplectic toric $(G \times G)$-manifold $$(M_1 \times M_2, \pi_1^* \omega_1 \oplus \pi_2^* \omega_2,\pi_1^* \mu_1 \oplus \pi_2^* \mu_2)$$ with respect to the $S/S_\ker$-action, where $S_\ker =S\cap(G\times G)_\ker$, at level zero is a symplectic toric $G$-manifold with moment polytope $\Delta_1 \cap \Delta_2$.
\end{proposition}

\begin{proof}
The zero level is
$Z = \{ (p_1,p_2) \in M_1 \times M_2 \mid \mu_1(p_1)=\mu_2(p_2) \}$. For an element $(p_1,p_2) \in Z$, the $S/S_\ker$-action is free if and only if
\[
S\cap (G\times G)_{(p_1,p_2)} = S\cap (G\times G)_\ker.
\]
Since $G\times G$ acts with the product action, it follows that 
\begin{equation*}
    S\cap (G\times G)_\ker = i_\skew(G_{\ker,1}\cap G_{\ker,2})
\end{equation*}
and
\begin{equation*}
    S\cap (G\times G)_{(p_1,p_2)} = i_\skew(G_{p_1}\cap G_{p_2}).
\end{equation*}
Since $i_\skew$ is an embedding, the action being free is thus equivalent to
\begin{equation*}
    G_{p_1} \cap G_{p_2} = G_{\ker,1} \cap G_{\ker,2}.
\end{equation*}
By \cref{cor:subtorusintersection}, this is equivalent to
the equality of annihilator lattices
\begin{equation*}
    (\fg_{p_1}^0 \cap \fgz^*) + (\fg_{p_2}^0 \cap \fgz^*) = (\fg_{\ker,1}^0 \cap \fgz^*) + (\fg_{\ker,2}^0 \cap \fgz^*).
\end{equation*}
Finally, it follows from Delzant's Lemma (\cref{cor:delzants_lemma}) that this is
the condition for integral transversal intersection of the polytopes
at $\xi := \mu_1(p_1)=\mu_2(p_2)$.
\end{proof}

\begin{remark}
    This construction pertains to a feature of usual transversality: Let $A_1,A_2\subset \fg^*$ be two subspaces. Then, if $D$ denotes the diagonal subspace in $\fg^* \times \fg^*$, we have
    \begin{equation*}
        A_1\pitchfork A_2 \qquad \iff \qquad (A_1 \times A_2)\pitchfork D.
    \end{equation*}
    This immediately carries over to relative transversality for polyhedra:
    \begin{equation*}
        P_1 \pitchfork P_1 \qquad \iff \qquad (P_1\times P_2) \pitchfork D.
    \end{equation*}
    Intersecting with $\fgz^*$ then gives that
    \begin{equation*}
        P_1 \pitchfork_{_\ZZ} P_1 \qquad \iff \qquad (P_1\times P_2) \pitchfork_{_\ZZ} D.
    \end{equation*}
    We refer to \cref{sec:polyehedronintersection} for a more detailed account.
\end{remark}


\section{Cutting, Chopping and Merging}
\label{sec:reduction_2}

Focusing on the intersection of a polytope
with a half-space,
we  address the classic constructions of symplectic cutting (\cref{sec:cutting_levels}),
its special cases of
the equivariant blow-up of a point or a submanifold --
corner and face chopping
(\cref{sec:choppingcorner,sec:choppingface}) --
as well as the converse of cutting
-- merging (\cref{sec:merging}) --
gluing two toric manifolds along a shared
toric submanifold whenever their normal bundles there
are equivariantly opposite.

\subsubsection*{Working assumption}

All the constructions in this section
preserve the dimensions of the polytopes.
Therefore, to simplify notation,
we assume  the kernel $G_\ker$ of the $2n$-dimensional
symplectic toric $G$-manifold is trivial, so $\dim(G)=n$.
(Otherwise, replace $G$ below with $G/G_\ker$ and
$H$ with $HG_\ker/G_\ker$.)


\subsection{Cutting Levels}
\label{sec:cutting_levels}

Let $(M,\omega,\mu)$ be a $2n$-dimensional symplectic toric $G$-manifold.
Let $H$ be a 1-dimensional subtorus of $G$ such that
$\fh^0 \subseteq \fg^*$ is a rational hyperplane. Recall that a translate $A$ of $\fh^0$ that is
an $H$-reduction level for $(M,\omega,\mu)$
is also called an $H$-\textit{cutting level} for $(M,\omega,\mu)$.
By \cref{prop:reduction_levels}, the criterion for $A$ to be
an $H$-cutting level is that $A \pitchfork_{_\ZZ} \Delta$,
which we assume from now on.

Let $\AA^+$ and $\AA^-$ be the closures of the two connected components
of $\fg^*\setminus A$, and set
\[
\Delta^+ := \Delta \cap \AA^+ \qquad \text{and} \qquad
\Delta^- := \Delta \cap \AA^-,
\]
called the two \textit{cut polytopes} of $\Delta$ by $A$.

By Lerman's symplectic cutting construction~\cite{LermanSymplecticCuts},
cutting $M$ along the $H$-cutting level $A$ produces two
symplectic cut spaces $M^{\cut,+}$ and $M^{\cut,-}$,
whose moment polytopes are exactly the two cut polytopes
$\Delta^+$ and $\Delta^-$.
The cut spaces may be viewed as compactifications of the open $G$-invariant submanifolds
\[
\mu^{-1} (\AA^+) \qquad \text{and} \qquad
\mu^{-1} (\AA^-)
\]
(see \cite[p.225]{ginzburgguilleminkarshon1996}
or \cite[p.253]{lerman_four_author}).
Since the quotient $\mu^{-1}(A)/H$ is smooth exactly when
both symplectic cut spaces of $(M,\omega,\mu)$
along $A$ are smooth,
we see that from $A\pitchfork_{_\ZZ}\Delta$ it follows that
both $\Delta^+$ and $\Delta^-$ are themselves unimodular polytopes.


\subsection{Corner Chopping}
\label{sec:choppingcorner}

The most elementary instance of cutting is the one where one of the
two resulting polytopes, say $\Delta^-$, is a unimodular simplex:
the corresponding cut space is then a scaled $\CC\PP^n$,
and the other cut space is a \textit{blow-up} of $(M,\omega,\mu)$
at a fixed point.

Concretely, $\Delta^-$ has exactly one vertex $\xi$ not lying in $A$.
If $\beta_1,\ldots,\beta_n\in\fgz^*$ are the primitive vectors
along the edges of $\Delta$ pointing away from $\xi$,
there is $\varepsilon>0$ such that
the remaining vertices of $\Delta^-$ are
$\xi+\varepsilon \beta_1,\ldots,\xi+\varepsilon \beta_n$.
The polytope $\Delta^+ =: \Delta_\xi^\varepsilon$ is called the
\textit{$\varepsilon$-blow-up of $\Delta$ at the vertex $\xi$},
and the corresponding cut space is the
\textit{$\varepsilon$-blow-up of $M$ at $\mu^{-1}(\xi)$}.

\medskip

We next show that this picture, together with the existence of a
suitable cutting hyperplane $A$, generalizes to any proper
(and non-empty) face $F$ of any dimension,
with an explicit description of $A$ as a hyperplane through
points obtained from the vertices of $F$ by heading out of $F$
by the same multiple $\varepsilon$ along primitive edge vectors.

\subsection{Face chopping}
\label{sec:choppingface}

Let $F$ be a $k$-dimensional face of $\Delta$, $0\leq k\leq n-2$.\footnote{The case of facets, i.e., $k=n-1$ is treated separately
in part 2 of \cref{rmk:chopping}.}
Let $F_1,\ldots,F_{n-k}$ be the facets of $\Delta$ containing $F$.
Thus, $F = F_1\cap\cdots\cap F_{n-k}$, as $\Delta$ is unimodular,
hence simple.
Let $w_1,\ldots,w_{n-k} \in \fgz$ be the corresponding
primitive \textit{inward conormals}, i.e.:
\begin{equation*}
F_i = \Delta \cap \partial\HH_{(w_i,c_i)}, \qquad
\text{with } c_i := \expval{\xi,w_i} \text{ for } \xi \in F_i .
\end{equation*}

Fix a vertex $\xi_j$ a vertex of $F$. Since $\Delta$ is unimodular, the primitive vectors along the edges of $\Delta$ form a $\ZZ$-basis of $\fgz^*$. The vectors of the dual basis define the $n$-facets which intersect at the vertex $\xi_j$. In particular, since $F=F_1\cap \dots \cap F_{n-k}$, this basis contains the vectors $w_1,...,w_{n-k}$. Let $w'_{1},...,w'_{k}\in \fgz$ be the remaining basis vectors so that the dual basis is $w_1,...,w_{n-k},w'_{1},...,w'_k\in \fgz$. We write $\beta_{j,1},...,\beta_{j,n-k}$ for the elements dual to the $w_i$ and $\alpha_{j,1},...,\alpha_{j,k}$ for the elements dual to the $w'_i$. By construction, the $\alpha_{j,1},...,\alpha_{j,k}$ generate precisely the edges contained in $F$ and the $\beta_{j,1},...,\beta_{n-k}$ generate the edges pointing away from $F$.
See \cref{fig:data_for_polytope_blowup_along_edge} for an illustration
of these vectors.

\begin{figure}[ht]
\begin{center}
\begin{tikzpicture}[tdplot_main_coords, scale=2.8]
				
				\coordinate (O) at (0,0,0);
				\coordinate (A) at (1.5,0,0); 
				\coordinate (B) at (0,1,0); 
				\coordinate (C) at (0,0,1); 
				\coordinate (AB) at (1.5,1,0);
				\coordinate (BC) at (0,1,1);
				
				\fill[red] (0,0,0) circle (1pt);
				\fill (1.5,0,0) circle (1pt);
				\fill (0,1,0) circle (1pt);
				\fill[red] (0,0,1) circle (1pt);
				\fill (1.5,1,0) circle (1pt);
				\fill (0,1,1) circle (1pt);
				
				\draw[-stealth,opacity=.5] (O) -- (2.2,0,0) node[below] {$x$};
				\draw[-stealth,opacity=.5] (O) -- (0,1.4,0) node[below] {$y$};
				\draw[-stealth,opacity=.5] (O) -- (0,0,1.4) node[left] {$z$};
				
				\draw[thick,dashed] (O) -- (A);
				\draw[thick,dashed] (O) -- (B);
				\draw[red,line width=2pt,densely dashed] (O) -- (C);
				\draw[thick] (B) -- (BC);
				\draw[thick] (C) -- (BC);
				\draw[thick] (A) -- (AB);
				\draw[thick] (B) -- (AB);
				\draw[thick] (AB) -- (BC);
				\draw[thick] (A) -- (C);
				
				\filldraw[fill=blue!20, opacity=0.5] (B) -- (BC) -- (AB) -- cycle;
				\filldraw[fill=blue!20, opacity=0.5] (C) -- (BC) -- (AB) -- (A) -- cycle;

                \node at (0,1.3,1.2) {$\Delta$};
                \node at (0,-0.1,0.4) {${\color{red}F}$};

                \node at (0,-0.15,0.1) {${\color{red}\xi_1}$};
                \node at (0,-0.15,1) {${\color{red}\xi_2}$};

    \draw[-stealth,thick] (C) -- (0.5,0,0.7) node[left] {$\beta_{2,1}$};
    \draw[-stealth,thick] (C) -- (0,0.3,1) node[above] {$\beta_{2,2}$};
    \draw[-stealth,thick] (C) -- (0,0,0.7) node[right] {$\alpha_2$};

    \draw[-stealth,thick] (O) -- (0.5,0,0) node[below right] {$\beta_{1,1}$};
    \draw[-stealth,thick] (O) -- (0,0.3,0) node[below right] {$\beta_{1,2}$};
    \draw[-stealth,thick] (O) -- (0,0,0.3) node[right] {$\alpha_1$};

\end{tikzpicture}
\caption{A 1-dimensional face $F$ (along the vertical axis) of the polytope $\Delta$
with its vertices ($m=2$) and corresponding edge vectors.
By our conventions, the facet $F_1$ is the rectangle on the $yz$-plane,
whereas the facet $F_2$ is the triangle on the $xz$-plane.}
\label{fig:data_for_polytope_blowup_along_edge}
\end{center}
\end{figure}
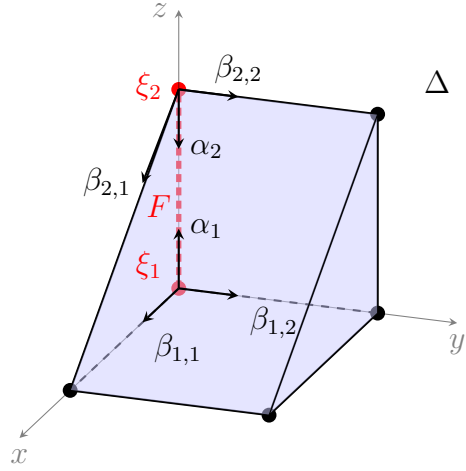

In summary, the cone $C_{\xi_j}\Delta$
has a $\ZZ$-basis
\begin{equation}
\label{eq:cone_basis}
\underbrace{\alpha_{j,1},\ldots,\alpha_{j,k}}_{\text{in $C_{\xi_j}F$}},
\underbrace{\beta_{j,1},\ldots,\beta_{j,n-k}}_{\text{not in $C_{\xi_j}F$}},
\end{equation}
so that the dual basis vector to $\beta_{j,i}$ is precisely $w_i$, the facet of $C_{\xi_j}\Delta$ opposite
$\beta_{j,i}$ is $C_{\xi_j}F_i$, whose primitive inward conormal is $w_i$, and
\begin{equation}
\label{eq:dualitychopping}
\expval{\alpha_{j,l},w_r}=0,
\quad
\expval{\beta_{j,i},w_r} = \delta_{ir} \quad (1\leq l\leq k, 1\leq i,r\leq n-k),
\end{equation}
where the first identity holds simply because $\alpha_{j,l}$ is tangent to
$F\subseteq F_r$.

\begin{proposition}[Integral Transversality for Face Chopping]
\label{prop:blowup_int_transv}
In the notation above and for $\varepsilon>0$, set
\begin{align*}
w_{_F} & := w_1+\cdots+w_{n-k}, \\
c_{_F} & := c_1+\cdots+c_{n-k}, \\
A_F^\varepsilon & := \{\xi\in\fg^* \mid \expval{\xi,w_{_F}}=c_{_F}+\varepsilon\} \qquad \text{ and} \\
\cV_F^\varepsilon & := \{ \xi_j+\varepsilon \beta_{j,i} \mid j=1,\ldots,m,\ \ i=1,\ldots,n-k \}.
\end{align*}
Then:
\begin{enumerate}
    \item
    $w_{_F}$ is a primitive element of $\fgz$;
    \item
    $\cV_F^\varepsilon$ lies on the affine hyperplane $A_F^\varepsilon$;
    \item
    for $\varepsilon$ sufficiently small,
    $A_F^\varepsilon$ intersects $\Delta$ integrally transversely
    and $A_F^\varepsilon\cap\Delta$ is the unimodular polytope
    with vertex set $\cV_F^\varepsilon$.
\end{enumerate}
\end{proposition}

\begin{proof}
\begin{enumerate}
    \item
    Fix $j$, and let $w_1',\ldots,w_k',w_1,\ldots,w_{n-k}$ be the
    $\ZZ$-basis of $\fgz$ dual to that in \eqref{eq:cone_basis}.
    In these coordinates, $w_{_F}$ has primitive coordinate vector
    $(0,\ldots,0,1,\ldots,1)$, hence is primitive.

    \item
    Since $\xi_j\in F\subseteq F_i$ for every $i=1,\ldots,n-k$,
    \eqref{eq:dualitychopping} gives, for every $j$ and every $i$,
    \[
    \expval{\xi_j+\varepsilon \beta_{j,i},w_{_F}}
    = \sum_{r=1}^{n-k}\expval{\xi_j,w_r}
    + \varepsilon \sum_{r=1}^{n-k}\expval{\beta_{j,i},w_r}
    = c_{_F}+\varepsilon.
    \]
    
    \item
    Since each $w_i$ is an inward conormal for $\Delta$, we have
    $\expval{\xi,w_i}\geq c_i$ for every $\xi\in\Delta$, so
    $\expval{\xi,w_{_F}}\geq c_{_F}$ throughout $\Delta$,
    with equality exactly when $\xi\in F_1\cap\cdots\cap F_{n-k}=F$.

    Whereas the vertices of $F$ satisfy $\expval{\xi,w_{_F}} - c_{_F} = 0$,
    the finitely many vertices of $\Delta$ outside $F$ satisfy
    $\expval{\cdot,w_{_F}}-c_{_F}\geq \delta_{\min}$ for some
    $\delta_{\min}>0$.
    Choose any positive $\varepsilon$ smaller than $\delta_{\min}$.

    Since $A_F^\varepsilon$ is an affine hyperplane
    and misses every vertex of $\Delta$,
    by \cref{cor:consequencesofaffinetransversality}, $A_F^\varepsilon$
    meets $\Delta$ transversely.
    By \cref{cor:consequencesofaffinetransversality},
    the vertices of $A_F^\varepsilon\pitchfork\Delta$
    are the points where $A$ intersects the edges of $\Delta$.
    
    Now every edge of $\Delta$ is of exactly one of three kinds:
    \begin{enumerate}
        \item an edge of $F$, tangent to $F$, along which
    $\expval{\cdot,w_{_F}}-c_{_F}$ vanishes identically by
    \eqref{eq:dualitychopping};
    \item one of the $m(n-k)$ edges from
    a vertex $\xi_j$ of $F$ in a direction $\beta_{j,i}$, whose
    points $\xi_j+t\beta_{j,i}$ satisfy
    $\expval{\xi_j+t\beta_{j,i},w_{_F}}-c_{_F}=t$ by \eqref{eq:dualitychopping},
    so the level $\varepsilon$ is crossed exactly once, at
    $t=\varepsilon$ corresponding to one of the points in $\cV_F^\varepsilon$;
    \item an edge with both endpoints outside $F$,
    along which the affine function $\expval{\cdot,w_{_F}}-c_{_F}$ stays
    $\geq\delta_{\min}>\varepsilon$ throughout.
    \end{enumerate}

    Therefore, $\text{Vert}(A_F^\varepsilon\cap\Delta)=\cV_F^\varepsilon$.

    It remains to check integral transversality.
    By \cref{cor:criteriaforaffinetransversality}, it is enough to check it
    at each vertex $\eta=\xi_j+\varepsilon \beta_{j,i}$ in $\cV_F^\varepsilon$.
    As $\eta$ lies in the relative
    interior of the edge in direction $\beta_{j,i}$, we have
    $L_\eta\Delta=\RR \beta_{j,i}$, so the integral transversality
    condition at $\eta$ for the direction subspace
    $B:=\partial\HH_{w_{_F}}$ of $A_F^\varepsilon$ reads
    \[
    (B\cap\fgz^*) + \ZZ \beta_{j,i} = \fgz^*.
    \]
    Indeed, this follows from $\expval{\beta_{j,i},w_{_F}}=1$:
    for any $y\in\fgz^*$, the element
    $y-\expval{y,w_{_F}}\,\beta_{j,i}$ lies in $B\cap\fgz^*$, so
    $y \in (B\cap\fgz^*)+\ZZ \beta_{j,i}$.
    
    By \cref{cor:int_transv_gives_unimodular},
    $A_F^\varepsilon\cap\Delta$ is thus a unimodular polytope. \qedhere
\end{enumerate}
\end{proof}

\begin{definition}
For $\varepsilon > 0$ sufficiently small as in part 3
of \cref{prop:blowup_int_transv},
the \textit{$\varepsilon$-blow-up of $\Delta$ along $F$}
is the cut polytope $\Delta_F^\varepsilon := \Delta \cap \AA^+$,
where $\AA^+$ is the closure of the connected component
of $\fg^* \setminus A_F^\varepsilon$ not containing
$F$.\footnote{In formulas, $\AA^+ =
\{\xi\in\fg^*\mid \expval{\xi,w_{_F}}\geq c_{_F}+\varepsilon\}$.}
\end{definition}

\Cref{fig:polytope_blowup_along_edge} illustrates
\cref{prop:blowup_int_transv} for $n=3$, $k=1$,
with $F$ the edge on the $z$-axis.
In this case, $\cV_F^\varepsilon$ has four points, all lying on the
plane $A_F^\varepsilon$ given by $x+y=\varepsilon$,
and $A_F^\varepsilon \cap \Delta$ is a Hirzebruch trapezoid.

        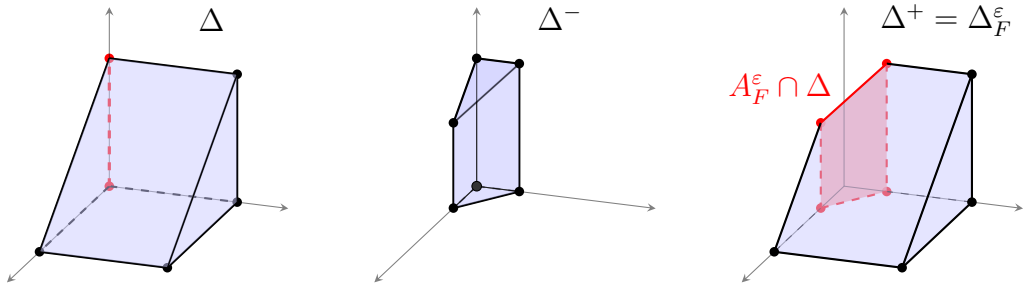
\begin{figure}[ht]
		\begin{tikzpicture}[tdplot_main_coords, scale=1.8]

            
			\begin{scope}[xshift=0cm]
				
				\coordinate (O) at (0,0,0);
				\coordinate (A) at (1.5,0,0); 
				\coordinate (B) at (0,1,0); 
				\coordinate (C) at (0,0,1); 
				\coordinate (AB) at (1.5,1,0);
				\coordinate (BC) at (0,1,1);
				
				\fill[red] (0,0,0) circle (1pt);
				\fill (1.5,0,0) circle (1pt);
				\fill (0,1,0) circle (1pt);
				\fill[red] (0,0,1) circle (1pt);
				\fill (1.5,1,0) circle (1pt);
				\fill (0,1,1) circle (1pt);
				
				\draw[-stealth,opacity=.5] (O) -- (2.2,0,0);
				\draw[-stealth,opacity=.5] (O) -- (0,1.4,0);
				\draw[-stealth,opacity=.5] (O) -- (0,0,1.4);
				
				\draw[thick,dashed] (O) -- (A);
				\draw[thick,dashed] (O) -- (B);
				\draw[red,very thick,dashed] (O) -- (C);
				\draw[thick] (B) -- (BC);
				\draw[thick] (C) -- (BC);
				\draw[thick] (A) -- (AB);
				\draw[thick] (B) -- (AB);
				\draw[thick] (AB) -- (BC);
				\draw[thick] (A) -- (C);
				
				\filldraw[fill=blue!20, opacity=0.5] (B) -- (BC) -- (AB) -- cycle;
				\filldraw[fill=blue!20, opacity=0.5] (C) -- (BC) -- (AB) -- (A) -- cycle;

                \node at (0,0.8,1.4) {$\Delta$};

			\end{scope}
	
           			
			\begin{scope}[xshift=2.7cm]
				
				\coordinate (O) at (0,0,0);
				\coordinate (A) at (1.5,0,0); 
				\coordinate (B) at (0,1,0); 
				\coordinate (C) at (0,0,1); 
				\coordinate (AB) at (1.5,1,0);
				\coordinate (BC) at (0,1,1);
						
	\coordinate (AC1) at (0.5,0,0.666666);
	\coordinate (BCC1) at (0,0.333333,1);
	\coordinate (A1) at (0.5,0,0);
	\coordinate (B1) at (0,0.333333,0);
					
				\draw[thick,dashed] (O) -- (B1);
				\draw[thick,dashed] (O) -- (A1);

                \fill[fill=blue!10, opacity=0.5] (AC1) -- (BCC1) -- (B1) -- (A1) -- cycle;

				\draw[thick] (BCC1) -- (B1) -- (A1) -- (AC1) -- cycle;

                \filldraw[fill=blue!20, opacity=0.5] (O) -- (B1) -- (A1) -- cycle;
                \filldraw[fill=blue!20, opacity=0.5] (AC1) -- (C) -- (O) -- (A1) -- cycle;
                \filldraw[fill=blue!20, opacity=0.5] (C) -- (BCC1) -- (B1) -- (O) -- cycle;
                \filldraw[fill=blue!20, opacity=0.5] (AC1) -- (BCC1) -- (C) -- cycle;
                
				\draw[-stealth,opacity=.5] (O) -- (2.2,0,0);
				\draw[-stealth,opacity=.5] (O) -- (0,1.4,0);
				\draw[-stealth,opacity=.5] (O) -- (0,0,1.4);

	\filldraw[fill=black!80] (0,0,0) circle (1pt);

	\fill (BCC1) circle (1pt);
	\fill (AC1) circle (1pt);				
	\fill (A1) circle (1pt);
	\fill (B1) circle (1pt);
	\fill (0,0,1) circle (1pt);
				

				\draw[thick] (BCC1) -- (C);
				\draw[thick] (C) -- (AC1);
	
                \node at (0,0.65,1.4) {$\Delta^-$};

			\end{scope}

           			
			\begin{scope}[xshift=5.4cm]
				
				\coordinate (O) at (0,0,0);
				\coordinate (A) at (1.5,0,0); 
				\coordinate (B) at (0,1,0); 
				\coordinate (C) at (0,0,1); 
				\coordinate (AB) at (1.5,1,0);
				\coordinate (BC) at (0,1,1);
						
	\coordinate (AC1) at (0.5,0,0.666666);
	\coordinate (BCC1) at (0,0.333333,1);
	\coordinate (A1) at (0.5,0,0);
	\coordinate (B1) at (0,0.333333,0);
					
	\fill[red] (A1) circle (1pt);
	\fill[red] (B1) circle (1pt);
			
	\fill[red] (BCC1) circle (1pt);
	\fill[red] (AC1) circle (1pt);				
				
	\fill[fill=red!40, opacity=0.8] (AC1) -- (BCC1) -- (B1) -- (A1) -- cycle;
				
				\draw[-stealth,opacity=.5] (O) -- (2.2,0,0);
				\draw[-stealth,opacity=.5] (O) -- (0,1.4,0);
				\draw[-stealth,opacity=.5] (O) -- (0,0,1.4);

				\draw[dashed] (B) -- (B1);
				\draw[dashed] (A) -- (A1);
				\draw[dashed,red,thick] (BCC1) -- (B1);
				\draw[dashed,red,thick] (A1) -- (B1);
				\draw[dashed,red,thick] (AC1) -- (A1);

	\filldraw[fill=blue!20, opacity=0.5] (B) -- (BC) -- (AB) -- cycle;
	\filldraw[fill=blue!20, opacity=0.5] (AC1) -- (BCC1) -- (BC) -- (AB) -- (A) -- cycle;			

				\draw[thick] (B) -- (BC);
				\draw[thick] (BCC1) -- (BC);
				\draw[thick] (A) -- (AB);
				\draw[thick] (B) -- (AB);
				\draw[thick] (AB) -- (BC);
				\draw[thick] (A) -- (AC1);
				\draw[red,thick] (BCC1) -- (AC1);

	\fill (1.5,0,0) circle (1pt);
	\fill (0,1,0) circle (1pt);
	\fill (1.5,1,0) circle (1pt);
	\fill (0,1,1) circle (1pt);

                \node at (0,0.8,1.4) {$\Delta^+=\Delta_F^\varepsilon$};
                \node[red] at (0,-0.5,0.7) {$A_F^\varepsilon \cap \Delta$};

			\end{scope}
			
		\end{tikzpicture}
		\caption{An $\varepsilon$-blow-up of the polytope $\Delta$ on the
        left along the edge on the $z$-axis produces a polytope
        $\Delta^+$ as on the right side.
        The two corresponding cut polytopes of $\Delta$ are
        $\Delta^-$ in the middle and $\Delta^+$ on the right.}
		\label{fig:polytope_blowup_along_edge}
	\end{figure}

\begin{corollary}[Face Chopping]
\label{cor:facechopping}
In the situation of \cref{prop:blowup_int_transv},
let $H\subseteq G$ be the circle subgroup with Lie algebra
$\RR w_{_F}\subseteq\fg$, so that $\fh^0=\partial\HH_{w_F}$ is the
direction subspace of $A_F^\varepsilon$.

Then, for every sufficiently small $\varepsilon>0$,$A_F^\varepsilon$ is an $H$-cutting level for
$(M,\omega,\mu)$, and $\Delta_F^\varepsilon$
is the unimodular polytope whose vertex set is
the union of $\cV_F^\varepsilon$ with the
set of vertices of $\Delta$ away from $F$.
\end{corollary}

\begin{proof}
By definition, $A_F^\varepsilon$ is a translate of $\fh^0$.
By part (3) of \cref{prop:blowup_int_transv}, $A_F^\varepsilon\pitchfork_{_\ZZ}\Delta$.
By \cref{prop:reduction_levels}, $A_F^\varepsilon$ is then
an $H$-reduction level,
which for these dimensions is
also called an $H$-cutting level, for $(M,\omega,\mu)$.

By the observation at the start of this subsection,
$\Delta_F^\varepsilon = \Delta\cap\AA^+$ is a
unimodular polytope (as well as $\Delta\cap\AA^-$,
the closure of $\fg^*\setminus \AA^+$).

By part 3 of \cref{prop:blowup_int_transv},
the vertices of $\Delta_F^\varepsilon$ lying in
$A_F^\varepsilon$ are the points in $\cV_F^\varepsilon$.
The vertices of $\Delta_F^\varepsilon$ lying in
the interior of $\AA^+$ are those of $\Delta$
lying in that halfspace, which are all but those on $F$.
\qedhere
\end{proof}

\begin{definition}
The \textit{$\varepsilon$-blow-up of $(M,\omega,\mu)$
along the toric submanifold $\mu^{-1}(F)$} is the
symplectic toric $G$-manifold
with moment polytope $\Delta_F^\varepsilon$,
where $F$ is a proper face of $\Delta=\mu(M)$
of codimension at least 2
and $\varepsilon>0$ is sufficiently small as in
\cref{prop:blowup_int_transv}.
\end{definition}

\begin{remarks}
\label{rmk:chopping}
$\phantom{x}$
\begin{enumerate}

\item
Taking $k=0$ recovers corner
chopping, with $w_{_F}$ the sum of \textit{all} $n$
primitive inward facet conormals at the vertex $F=\{\xi\}$.

\item
When $k=n-1$ the previous procedure degenerates.
The formula for $w_{_F}$ still makes sense, but the construction
reduces to ordinary cutting along a hyperplane parallel
to the facet $F$, rather than to a genuine blow-up.
In particular, it does not change the topology.

\item
Geometrically, $\Delta\cap\AA^-$,
i.e., the small polytope chopped off near $F$,
is the moment polytope of the total space of the projectivization of the
normal bundle of $\mu^{-1}(F)$ in $M$ plus a trivial $\CC$.
In the special case when $F$ is a point
$p$ ($k=0$),
that small polytope is a simplex and the total
space of $\PP (T_pM \oplus \CC)$ 
is a (scaled) $\CC\PP^n$.
\end{enumerate}
\end{remarks}


\subsection{Merging}
\label{sec:merging}

We now turn to the reverse of cutting,
namely \textit{merging} two unimodular polytopes together along a common facet,
so that the facet is fully absorbed into a cutting level of the merged polytope.

Let $\Delta^+,\Delta^-\subseteq \fg^*$ be unimodular polytopes of
dimension $n$ such that
\[
   F:=\Delta^+ \cap \Delta^-
\]
is a facet of each of $\Delta^+$ and $\Delta^-$.
Then $B := \dir{F}$ is an $(n-1)$-dimensional subspace of $\fg^*$.
Let $A$ be the translate of $B$ containing $F$.
Each $\Delta^+$ and $\Delta^-$ lies in one of the two closed half-spaces
determined by $A$.
Let $v\in \fgz$ be the primitive normal with
$\Delta^+ \subseteq \HH_{(v,c)}$, hence
$\Delta^- \subseteq \HH_{(-v,-c)}$.

Fix a vertex $\eta\in \text{Vert}(F)$. By \cref{lem:cleanintersection}, we have $$C_\eta F = C_\eta(\Delta^+\cap \Delta^-)=C_\eta\Delta^+\cap C_\eta\Delta^-,$$ which shows that the local cones at $\eta$ also share a facet. It follows that the two bases provided by the edges of $C_\eta\Delta^+$ and $C_\eta \Delta^-$ respectively have $n-1$ elements in common. We will write these vectors as $e_1,...,e_{n-1}\in \fgz^*$ and the remaining edge of $C_\eta\Delta^\pm$ as $e_n^{(\pm)}$ respectively.

\begin{definition}
\label{def:mergeable}
In the above situation, we say that $\Delta^+$ and $\Delta^-$ are
\textit{mergeable along $F$} if, for every vertex
$\eta \in \text{Vert}(F)$,
\[
   e_n^{(-)} = -e_n^{(+)}.
\]
\end{definition}

\begin{lemma}
\label{lem:merging_local}
    Let $\eta \in \text{Vert}(F)$ and let
    $e_1,\ldots,e_{n-1},e_n$ be a $\ZZ$-basis of $\fgz^*$ as above,
    with $e_n^{(+)}=e_n$ and $e_n^{(-)}=-e_n$.
    Then
    \[
        C_\eta \Delta^+ \cup C_\eta \Delta^-
        = \co{e_1,\ldots,e_{n-1}} \oplus \RR e_n.
    \]
\end{lemma}

\begin{proof}
Since $e_1,\ldots,e_n$ is a basis, every $\xi \in \fg^*$ is
uniquely $\xi =\sum_{i=1}^{n-1} c_i e_i + c_n e_n$.
By simplicity,
    \[
    \xi \in C_\eta \Delta^+= \co{e_1,\ldots,e_{n-1},e_n}
    \iff \text{ all }
    c_i\geq 0 \; (i=1,\ldots,n),
    \]
and
    \[
    \xi \in C_\eta \Delta^- = \co{e_1,\ldots,e_{n-1},-e_n}
    \iff \text{ all }
    c_i \geq 0 \text{ for } i<n \text{ and } c_n\leq 0.
    \]
The union is therefore the set of points $\xi$
with $c_1,\ldots,c_{n-1}\geq 0$ and $c_n$ unrestricted,
which is $\co{e_1,\ldots,e_{n-1}} \oplus \RR e_n$.
\end{proof}


\begin{proposition}[Merging Polytopes]
\label{prop:merging_polytope}
    Let $\Delta^+,\Delta^- \subseteq \fg^*$ be unimodular polytopes
    with a common facet $F = \Delta^+ \cap \Delta^-$
    and mergeable along $F$.  Set
    \[
        \Delta := \Delta^+ \cup \Delta^-.
    \]
    Then
    \begin{enumerate}
        \item
        $\Delta$ is a polytope;
        \item
        $\Delta$ is unimodular and
    \[
        \text{Vert}(\Delta)
        = \big(\text{Vert}(\Delta^+)\cup \text{Vert}(\Delta^-)\big) \setminus \text{Vert}(F);
    \]
    \item
    the affine hyperplane $A$ containing $F$
    intersects $\Delta$ integrally transversely.
    \end{enumerate}
\end{proposition}

\begin{proof}
\begin{enumerate}
    \item
    We first show local convexity along $F$.

By \cref{lem:merging_local}, at every $\eta\in \text{Vert}(F)$,
the set $\Delta^+ \cup \Delta^-$ agrees, near $\eta$, with the
translate of a convex cone that is not pointed.\footnote{A \textit{pointed cone} is a polyhedral cone that contains no full line.
The local cones at a vertex of a polytope are pointed, whereas the local
cones at any other point of the polytope are not.}
Hence, $\Delta$ is convex near $\eta$
and $\eta$ is not a vertex of $\Delta$.

Let $E$ be the face of $F$ with $\xi\in\text{relint}(E)$,
and let $m:=\dim E$.
Fix a vertex $\eta\in\text{Vert}(E)\subseteq \text{Vert}(F)$.
Let $e_1,\ldots,e_{n-1},e_n$ be as in \cref{lem:merging_local}.
Since $F$ is unimodular, after reindexing we may
assume that $C_\eta E=\co{e_1,\ldots,e_m}$.
Arguing as in Step~1 of the proof of \cref{cor:int_transv_gives_unimodular}
(applied to $\Delta^+$ and its face $E$, and again to $\Delta^-$ and $E$),
we get
\begin{align*}
C_\xi\Delta^+ &= \spn{\RR}{e_1,\ldots,e_m} + \co{e_{m+1},\ldots,e_{n-1},e_n}, \\
C_\xi\Delta^- &= \spn{\RR}{e_1,\ldots,e_m} + \co{e_{m+1},\ldots,e_{n-1},-e_n}.
\end{align*}
By the same computation as in the proof of \cref{lem:merging_local}
(applied to the shorter list $e_{m+1},\ldots,e_{n-1}$ in place of
$e_1,\ldots,e_{n-1}$ there), we obtain
\begin{align*}
& \co{e_{m+1},\ldots,e_{n-1},e_n} \, \cup \, \co{e_{m+1},\ldots,e_{n-1},-e_n} \\
= \phantom{x} & \co{e_{m+1},\ldots,e_{n-1}} \oplus \RR e_n,
\end{align*}
so
\begin{align*}
    &C_\xi\Delta^+\cup C_\xi\Delta^- \\
= &\spn{\RR}{e_1,\ldots,e_m} + \co{e_{m+1},\ldots,e_{n-1}} \oplus \RR e_n
\end{align*}
is convex.

At points of $\Delta$ away from $F$,
local convexity follows from
local convexity of $\Delta^+$ and $\Delta^-$, since,
if $\xi\in\Delta^+\setminus F$ (respectively $\Delta^- \setminus F$),
there is always some ball centered at $\xi$
missing $\Delta^-$ (respectively $\Delta^+$).

The set $\Delta$ is closed and connected (the union
of two closed, convex, hence connected, sets with $F\neq\emptyset$ in common).
By the Tietze--Nakajima theorem
(see, for instance,~\cite{bjorndahl2010tietze}), a closed,
connected, locally convex subset of a euclidean space is convex.
Hence, $\Delta$ is convex.

Since $\Delta^\pm=\text{Conv}(\text{Vert}(\Delta^\pm))$,
we get $\Delta \subseteq \text{Conv}\bigl(\text{Vert}(\Delta^+)\cup
\text{Vert}(\Delta^-)\bigr)$.
Conversely, $\Delta$ is convex and contains
$\text{Vert}(\Delta^+)\cup\text{Vert}(\Delta^-)$, so it
contains the convex hull of this finite set.
Hence,
\[
\Delta = \text{Conv}\bigl(\text{Vert}(\Delta^+)\cup\text{Vert}(\Delta^-)\bigr),
\]
so $\Delta$ is a polytope by the Weyl--Minkowski theorem.

\item 
Since the vertices of $F$ are not vertices of $\Delta$,
from the above we conclude that
\[
    \text{Vert}(\Delta) \subseteq
    \big(\text{Vert}(\Delta^+)\cup \text{Vert}(\Delta^-)\big)
    \setminus \text{Vert}(F).
\]
We now show equality and unimodularity.
At any vertex of $\Delta^+$ away from $F$,
the local cone of $\Delta$ agrees with that of $\Delta^+$,
since the gluing does not affect the polytope away from $F$.
Therefore, this remains a vertex of $\Delta$
with a unimodular local cone inherited from $\Delta^+$.
Likewise for $\Delta^-$.


\item 
First,
$A\cap\Delta = (A\cap\Delta^+)\cup(A\cap\Delta^-) = F\cup F = F$.
By \cref{cor:criteriaforaffinetransversality}, it suffices to check
integral transversality at each vertex $\eta$ of $A\cap\Delta=F$.
The direction subspace of $A$ is $B=\spn{\RR}{e_1,\ldots,e_{n-1}}$
and satisfies
$B\cap\fgz^*=\ZZ e_1+\cdots+ \ZZ e_{n-1}$
(by construction of $e_1,\ldots,e_{n-1}$ as above).
Since $C_\eta\Delta = \co{e_1,\ldots,e_{n-1}} \oplus \RR e_n$,
we have $L_\eta\Delta = \RR e_n$,
so $L_\eta\Delta \cap \fgz^* = \ZZ e_n$.
On the other hand, $\dir{\Delta}=\fg^*$,
so $\dir{\Delta}\cap\fgz^*=\fgz^*$.
Since $e_1,\ldots,e_{n-1},e_n$ is a $\ZZ$-basis of $\fgz^*$,
\begin{align*}
    (B\cap\fgz^*)+(L_\eta\Delta\cap\fgz^*)
&= \ZZ e_1+\cdots+\ZZ e_n \\
&= \fgz^* \\
&= (B\cap\fgz^*)+(\dir{\Delta}\cap\fgz^*),
\end{align*}
which is precisely \cref{def:int_transv} at $\eta$.
So $A$ intersects $\Delta$ integrally transversely.
\qedhere
\end{enumerate}
\end{proof}

\begin{definition}
\label{def:polytope_merging}
The polytope $\Delta=\Delta^+\cup\Delta^-$ from \cref{prop:merging_polytope}
is called the \textit{merging of $\Delta^+$ and $\Delta^-$ along $F$}.
\end{definition}

On the manifold side, let $(M^+,\omega^+,\mu^+)$ and
$(M^-,\omega^-,\mu^-)$ be the symplectic toric $G$-manifolds
associated with $\Delta^+,\Delta^-$ by Delzant's Theorem
(\cref{thm:delzant}), and let $N$ be the symplectic toric $G$-manifold
associated with $F$. By \cref{sec:submanifolds},
\[
   N^+ := (\mu^+)^{-1}(F) \subseteq M^+
   \qquad \text{and} \qquad
   N^- := (\mu^-)^{-1}(F) \subseteq M^-
\]
are symplectic toric submanifolds, both isomorphic to $N$, and each of
real codimension $2$ in its ambient manifold (since $F$ is a facet).
For $\eta \in \text{Vert}(F)$ and $p_\eta \in N$ the corresponding
fixed point, the symplectic slice representations of $N^+$ in $M^+$
and of $N^-$ in $M^-$ at $p_\eta$ are the $1$-dimensional toric
representations with weights $e_n^{(+)}(\eta)$ and $e_n^{(-)}(\eta)$,
respectively (\cref{sec:representations}).
Mergeability along $F$ is exactly the condition that
these weights are the negatives of one another
at every fixed point of $N$.

\cref{prop:merging_manifold} will be
a toric analogue of Gompf's symplectic sum~\cite{gompf1995construction}.

\begin{lemma}[Mergeability in terms of Normal Bundles]
In the above situation, let $\nu^+\to N$ and $\nu^-\to N$
be the normal bundles of $N^+\subseteq M^+$ and $N^-\subseteq M^-$.
Then  mergeability along $F$ is equivalent to $\nu^+$ and $\nu^-$
being $G$-equivariantly antiisomorphic,
\begin{equation}
\label{eq:antiisom}
    \nu^- \;\cong_G\; (\nu^+)^*,
\end{equation}
\end{lemma}

\begin{proof}
Since $N^\pm$ is a codimension 2 symplectic submanifold,
$\nu^\pm$ is a rank-$2$ symplectic vector bundle.
Since the weight of the $H$-action on its fiber
is the single character $e_n^{(\pm)}$ throughout
(the same conormal $v$ was fixed for all of $F$ at the outset),
$\nu^\pm$ is the complex line bundle associated with the
$H$-principal bundle $\mu^{-1}(F)\to N$ via $\chi_{e_n^{(\pm)}}$.
The $G$-invariant compatible complex structure and norm this carries
(\cref{lem:classificationtorusreps}) make it a $G$-equivariant
Hermitian line bundle, on which $H$ acts fiberwise with weight
$e_n^{(\pm)}$, respectively $e_n^{(-)}=-e_n^{(+)}$ (\cref{lem:merging_local}).
Since a $G$-equivariant Hermitian line bundle on the
(compact) toric manifold $N$ is determined,
up to isomorphism, by its isotropy weight at each fixed point of $N$
(\cref{lem:delzantlocalmodel}), mergeability along $F$
(\cref{def:mergeable}) says precisely that $\nu^- \;\cong_G\; (\nu^+)^*$.
\end{proof}

\begin{proposition}[Merging Toric Manifolds]
\label{prop:merging_manifold}
Let $(M^\pm,\omega^\pm,\mu^\pm)$ be two symplectic toric $G$-manifolds such that
$\Delta^+ := \mu^+ (M^+)$ and $\Delta^- := \mu^- (M^-)$ are
mergeable along a facet $F$.
Let $N$ be the symplectic toric $G$-manifold associated with $F$,
and let \(N^\pm := (\mu^\pm)^{-1}(F) \subseteq M^\pm\)
be symplectic toric submanifolds, both isomorphic to $N$.
Then there is a symplectic toric $G$-manifold
    $(M,\omega,\mu)$, unique up to isomorphism, obtained by gluing
    $M^+ \setminus N^+$ to $M^-\setminus N^-$ along a neighborhood of
    $N$ using the toric local models of $N^+$ and $N^-$, such that $\mu(M) = \Delta^+\cup\Delta^-$.
\end{proposition}

\begin{proof}
$\Delta := \Delta^+\cup\Delta^-$ is unimodular by
\cref{prop:merging_polytope}.
By Delzant's Theorem (\cref{thm:delzant}),
there is a symplectic toric $G$-manifold $(M,\omega,\mu)$
with $\mu(M)=\Delta$, unique up to isomorphism.
It remains to identify $M$ with the space obtained by gluing
$M^+\setminus N^+$ to $M^-\setminus N^-$ along a neighborhood of $N$.

\medskip
\noindent \textit{Claim. The hyperplane $A$ is a cutting level for $M$.}

\noindent Let $H\subseteq G$ be the circle subgroup with $\fh^0=B=\dir F$ and let $A$ be the translate of $\fh^0$ containing $F$.
By \cref{prop:merging_polytope}, $A \pitchfork_{_\ZZ} \Delta$,
so by \cref{prop:reduction_levels} $A$ is an
$H$-reduction level for $M$, i.e., an $H$-cutting level
in the sense of \cref{sec:cutting_levels}.

\medskip
\noindent \textit{Claim. Cutting $M$ along $A$ recovers $M^+$ and $M^-$.}

\noindent Since $\Delta^+\subseteq \AA^+$ and $\Delta^-\subseteq \AA^-$ by
construction (with $\AA^\pm$ the two closed half-spaces determined by
$A$, as in \cref{sec:cutting_levels}),
and $\Delta=\Delta^+\cup \Delta^-$, we have
$\Delta\cap\AA^\pm=\Delta^\pm$. By Delzant's Theorem (\cref{thm:delzant}), the toric manifolds obtained by the cut are thus $M^+$ and $M^-$ respectively.

\medskip
\noindent \textit{Claim. Away from $N$, this is the asserted gluing.}

\noindent Away from the cut locus $N^\pm$, the symplectic cuts are tautological:
$M\setminus \mu^{-1}(F)$ is, $G$-equivariantly and compatibly with the
moment maps, the disjoint union of the open dense subsets
given by the complements of $N^{\pm}$ in each
of the two cut spaces, identified with $M^{+}$ and $M^{-}$ respectively
\cite{LermanSymplecticCuts}. This identifies $M\setminus \mu^{-1}(F)$,
$G$-equivariantly and compatibly with the moment maps,
with $(M^+\setminus N^+)\sqcup(M^-\setminus N^-)$.

\medskip
\noindent \textit{Claim. Near $N$, this extends across via
the antiisomorphism of normal bundles
$\nu^+$ and $\nu^-$, by mergeability.}

\noindent Fix a $G$-invariant Hermitian metric $h$ on $\nu^+$
(average any metric over the compact group $G$).
The musical isomorphism $\nu^+ \to (\nu^+)^*$,
$v \mapsto h(v,-)$, is then a $G$-equivariant, norm-preserving,
fiberwise conjugate-linear diffeomorphism.
Composing the musical isomorphism with \eqref{eq:antiisom}
gives a $G$-equivariant, norm-preserving diffeomorphism
\(
    \Theta: \nu^+ \longrightarrow \nu^-
\)
covering the identity on $N$, equivariant for the weight-reversal
$e_n^{(+)}\longrightarrow e_n^{(-)}=-e_n^{(+)}$.

Fix $\delta>0$ small enough that the toric local models of $N^+$ and
$N^-$ (\cref{def:localmodel}) identify equivariant tubular
neighborhoods of $N^+$ in $M^+$ and of $N^-$ in $M^-$ with the
$\delta$-disk bundles $D_\delta(\nu^+)$ and $D_\delta(\nu^-)$, via
embeddings under which $\mu^+$ and $\mu^-$ take the form
\[
    \mu^+(v) = \mu_{_N}(\pi(v)) + |v|_h^2\, e_n^{(+)}, \qquad
    \mu^-(w) = \mu_{_N}(\pi(w)) + |w|_h^2\, e_n^{(-)},
\]
where $\mu_{_N}:N\to F$ is the moment map of $N$
and $\pi$ denotes either of the disk bundle projections.
Removing the closed disk bundles
$D_{\delta/2}(N^+)\subset M^+$ and $D_{\delta/2}(N^-)
\subset M^-$ and gluing the two boundaries $S_\delta(\nu^+)$,
$S_\delta(\nu^-)$ via $\Theta$ produces a smooth, compact manifold
\[
    M' := \bigl(M^+\setminus D_{\delta/2}(N^+)\bigr)
    \ \cup_{\Theta}\
    \bigl(M^-\setminus D_{\delta/2}(N^-)\bigr),
\]
with a $G$-action induced from those on $M^+$ and $M^-$;
this is well defined since $\Theta$ is $G$-equivariant.
On the overlap annulus $\delta/2<|v|<\delta$, the two moment map
formulas above disagree, but since $\Theta$ is
norm- and weight-compatible, they differ only in the
transverse $H$-direction and by construction extend to a single smooth
$G$-invariant function on $M'$ restricting to $\mu^+$
and to $\mu^-$ away from the neck.
The same argument from \cite[p.249]{LermanSymplecticCuts} (following
\cite{gompf1995construction} and at the origin
of the name \textit{cutting})
produces a symplectic form $\omega$
and moment map $\mu$ on $M'$ with $\mu=\mu^+$
on $M^+\setminus D_\delta(N^+)$, $\mu=\mu^-$ on
$M^-\setminus D_\delta(N^-)$, and
\[
    \mu(M') = \Delta^+ \cup \Delta^- = \Delta.
\]
Thus $(M',\omega,\mu)$ is a symplectic toric $G$-manifold with moment
polytope $\Delta$, obtained by gluing $M^+\setminus N^+$ to $M^-
\setminus N^-$ along a neighborhood of $N$ using the toric local
models of $N^+$ and $N^-$, as claimed.
By \cref{thm:delzant}, $M'$ is isomorphic to the manifold $M$ of Step~1,
which is therefore the manifold asserted by the proposition.
Uniqueness up to isomorphism is also inherited from \cref{thm:delzant}.
\end{proof}

\begin{remark}
The proof above shows that the two cut polytopes of $\Delta$ by $A$
are $\Delta^+$ and $\Delta^-$.
Hence, merging can indeed be seen as an inverse construction of cutting.
\end{remark}

\appendix
\section{Polyhedral Cones and Polytopes}
\label{sec:polyhedral_cones_and_polytopes}

This appendix collects, in as self-contained form as possible, the convex-geometric
background used throughout the paper. Since our moment polytopes need
not span all of $\fg^*$ once nontrivial kernels are allowed, we develop
polyhedra, faces, and the local cone at a point 
directly for subsets that need not be full-dimensional, rather than
simply citing the standard full-dimensional theory.
We extend the usual notions
of simple and unimodular polytopes to this setting, before recording how both properties pass to faces -- facts used silently throughout \cref{sec:stm_framework,sec:reduction_1,sec:reduction_2}.
We define the central notions of relative and integral transversality for intersections of polyhedra (\cref{def:relativetransversality}). Two special cases thereof, namely affine slices and polytope intersections, are key in \cref{sec:reduction_1}. In these two cases, we prove that integral transversality preserves unimodularity (\cref{prop:transverseslices,cor:polytopeintersection}), results that might be seen as the polytope analogue of symplectic reduction and symplectic cutting. 


\subsection{Polyhedra}
\label{sec:polyhedra}

Let $\fg$ be a real vector space and let $\fgz$ be a lattice in $\fg$. We denote by $\fg^*$ and $\fgz^*$ the dual vector space and the dual lattice, respectively.

A {\em polyhedron}, denoted $P\subseteq \fg^*$, is the intersection
of finitely many affine closed halfspaces.
Explicitly, if we denote the halfspaces by
\begin{equation*}
    \HH_{(v,c)}=\{\xi\in \fg^*\mid \expval{\xi,v}\geq c\}
\end{equation*}
where $v\in \fg$ and $c\in \RR$, a polyhedron is a set of the form
\begin{equation*}
    P=\bigcap_{i\in \Ii} \HH_{(v_i,c_i)}
\end{equation*}
where $\Ii$ is a finite index set. A polyhedron $P$ is called \emph{rational} if $v_i\in \fgz$ for all $i\in \Ii.$ 

\subsubsection*{Polyhedral Cones}
A \textit{polyhedral cone}, denoted $C\subseteq \fg^*$, is a polyhedron such that each affine halfspace is actually a linear halfspace (i.e., $c=0$ in its defining inequality). In this case, we will use the shorter notation $\HH_{v}$ for the corresponding linear halfspace, that is, 
\begin{equation*}
    C=\bigcap_{i\in \Ii} \HH_{v_i}.
\end{equation*}
A polyhedral cone can equivalenty be characterised as the conic hull of a finite set of points in $\fg^*$ (see, for instance, \cite[Theorem 1.3]{ziegler2012lectures}). 

\subsubsection*{Polytopes}
A \textit{polytope}, denoted $\Delta\subseteq \fg^*$, is a bounded polyhedron. Equivalently, a polytope is the convex hull of a finite set of points in $\fg^*$. This equivalence is known as the Weyl-Minkowski theorem
(see, for instance,~\cite[Theorem 1.1]{barvinok2008integer}).

\subsubsection*{Faces of a Polyhedron}
A {\em face} of a polyhedron $P$ is a subset of the form
$F=P \cap \partial \HH_{(v,c)}$, where $\HH_{(v,c)}$ is an affine halfspace containing $P$. By convention, the empty set and $P$ itself are also considered to be faces. Any polyhedron is the disjoint union of the relative interiors of its faces. The faces of dimension zero are called \textit{vertices}, and we will denote the set of vertices of a polyhedron $P$ by $\text{Vert}(P)$. The faces of dimension and codimension one are called edges and facets respectively. 

\subsubsection*{Local Structure of a Polyhedron}
Let $\xi\in P$. The \textit{local cone} $C_\xi P$ at $\xi$ (also known as the \textit{cone of feasible directions}
of $P$ at $\xi$;
see, for instance,~\cite[Definition 6.1]{barvinok2008integer})
is the conic hull of the translation of $P$ by $-\xi$:
\begin{equation*}
    C_\xi P = \text{Cone}(P-\xi).
\end{equation*}
Let $F$ be the unique face of $P$ containing $\xi$ in its interior. The {\em local linear space} $L_\xi P$ at $\xi$ is the linear hull of the translation of $F$ by $-\xi$:  
\begin{equation*}
    L_\xi P = \spn{\RR}{F-\xi}.
\end{equation*}
Note that $\dims{F}=\dims{L_\xi P}$ and $L_\xi P = C_\xi P\cap (-C_\xi P)$. The local linear space at an interior point is called the \textit{direction space} of $P$ and is denoted by $\dir{P}$.
By construction, $P$ lies in a translate of its direction space. 

The subset of indices
\[
    \Ii_\xi(P) := \{i\in \Ii \mid \expval{\xi,v_i}=c_i\},
\]
is known as the set of inequalities of $P$ which is \emph{active} at $\xi$. Using this, we can write explicit expressions for the local structures:
\begin{lemma}
\label{lem:explicitexpressions}
    Let $P=\bigcap_{i\in \Ii}\HH_{(v_i,c_i)}$ be a polyhedron and $\xi\in P$. Then
    \begin{enumerate}
        \item $L_\xi P = \bigcap_{i\in \Ii_\xi(P)} \partial\HH_{v_i},$
        \item $(L_\xi P)^{0} = \spn{\RR}{\{ v_i\}_{i\in \Ii_\xi(P)}}$ and
        \item $C_\xi P = \bigcap_{i\in \Ii_\xi(P)} \HH_{v_i}.$
    \end{enumerate}
\end{lemma}
\begin{proof}
    See \cite[Lemma 1.3.1 and Lemma 1.3.3]{kaufmann2022polytopes}.
\end{proof}

\subsubsection*{Local Structures vs Global Structures}
From \cref{lem:explicitexpressions} we immediately deduce the following correspondence between faces of the local cone and faces of the polyhedron.
\begin{corollary} \label{cor:facecorrespondence}
    Let $P\subset \fg^*$ be a polyhedron and $\xi\in P$ be a point. The maps
    \begin{align*}
        \{\text{Faces of $P$ containing $\xi$}\} &\longrightarrow \{\text{Faces of $C_\xi P$}\} \\
        F &\longmapsto C_\xi F \\
        P\cap(\xi +K)&\mapsfrom K
    \end{align*}
    are mutually inverse and inclusion-preserving.
\end{corollary}
If the polyhedron is bounded, meaning that it is a polytope, it follows from the Minkowski-Weyl theorem that the polytope can be entirely reconstructed from the local cones at the vertices.
\begin{corollary} \label{cor:polytopeasconeintersection}
    Let $\Delta\subset \fg^*$ be a polytope and let $\textup{Vert}(\Delta)$ be the set of its vertices. Then
\begin{equation*}
    \Delta = \bigcap_{\xi \in \textup{Vert}(\Delta)} (\xi + C_\xi\Delta).
\end{equation*}
\end{corollary}

Combining these two observations, we obtain the characterization of the faces of polytopes in terms of the local cones, which is used in the main text (\cref{sec:submanifolds}):

\begin{lemma}
\label{lem:coneface_to_polytopeface}
Let $\Delta\subseteq \fg^*$ be a polytope, $\mathcal{V}\subseteq \textup{Vert}\left( \Delta \right)$ a subset of vertices, and $\mathcal{F}= \textup{Conv}(\mathcal{V})$ its convex hull. Then the following are equivalent:
\begin{enumerate}
    \item $\mathcal{F}$ is a face of $\Delta$;
    \item $C_\xi(\mathcal{F})$ is a face of $C_\xi(\Delta)$ for every $\xi\in \mathcal{V}$.
\end{enumerate} 
\end{lemma}

\begin{proof}
If $\mathcal{V}=\emptyset$, the result trivially holds. We thus assume that $\mathcal{V}$ is non-empty.
\begin{description}
  \item[1. $\mathbf{\Longrightarrow}$ 2.] Follows immediately from \cref{cor:facecorrespondence}.

\item[2. $\mathbf{\Longrightarrow}$ 1.] By \cref{cor:facecorrespondence}, for each $\xi \in \mathcal{V}$, the face $C_\xi\mathcal{F}$ of $C_\xi\Delta$ corresponds to a unique face $F_\xi$ of $\Delta$ given by $F_\xi = \Delta\cap (\xi + C_\xi\mathcal{F}).$ Then define
\begin{equation*}
    F=\bigcap_{\xi \in \mathcal{V}}F_\xi.
\end{equation*}
$F$ is a face of $\Delta$ since it is the intersection of faces of $\Delta.$ But 
\begin{align*}
    F&= \bigcap_{\xi\in \mathcal{V}}F_\xi \\
    &= \bigcap_{\xi\in \mathcal{V}}(\Delta \cap (\xi + C_\xi \mathcal{F}) \\
    &= \Delta \cap \left( \bigcap_{\xi \in \mathcal{V}}(\xi + C_\xi\mathcal{F})\right)  \\
    &= \Delta \cap \mathcal{F}  \\
    &= \mathcal{F}, 
\end{align*}
where, passing to the second last line, we use \cref{cor:polytopeasconeintersection}. \qedhere
\end{description}
\end{proof}

\subsubsection*{Preimages of Polyhedra}
For reference in the main text (\cref{sec:local_model}), we record the following elementary property of polyhedra and linear maps.
\begin{lemma} \label{lem:preimage}
    Let $i:\fh \hookrightarrow \fg$ be the inclusion of a subspace and $i^*:\fg^* \to \fh^*$ be the dual projection. The preimage of a polyhedron
    \begin{equation*}
        P = \bigcap_{j\in \Ii} \HH_{(v_j,c_j)} \subseteq \fh^*
    \end{equation*}
    under the projection $i^*$ is the polyhedron
    \begin{equation*}
        (i^*)^{-1}(P) = \bigcap_{j\in \Ii}\HH_{(i(v_j),c_j)} \subseteq \fg^*.
    \end{equation*}
    Moreover, for $\xi\in (i^*)^{-1}(P)$ it holds that $\Ii_\xi\left((i^*)^{-1}(P)\right) = \Ii_{i^*(\xi)}(P).$ 
\end{lemma}

\begin{proof}
    Both statements follow immediately from the tautological relation
    \begin{equation*}
        \expval{i^*(\xi),v} = \expval{\xi,i(v)}
    \end{equation*}
    for any $\xi\in \fg^*$ and any $v\in \fh.$
\end{proof}
    
\begin{remark}
\label{cor:tangenttopreimage}
    Combining \cref{lem:explicitexpressions} and \cref{lem:preimage}, we observe that
    \begin{align*}
        i\left( (L_{i^*(\xi)}(P))^0\right) &= i\left(\spn{\RR}{\{ v_j\}_{j\in \Ii_{i^*(\xi)}(P)}}\right) \\
        &= \spn{\RR}{\{ i(v_j)\}_{j\in \Ii_{\xi}((i^*)^{-1}(P))}} \\
        &= L_\xi \left((i^*)^{-1}(P)\right)^{0}.
    \end{align*}
\end{remark}


\subsection{Simple and Unimodular Polytopes}
\label{sec:polyhedral_cones}
We first generalize the usual notions of \textit{simple polyhedral cone} (see, for instance, \cite[p. 78]{barvinok2008integer}) and \textit{unimodular polyhedral cone} (see, for instance, \cite[Definition 13.6]{barvinok2008integer}) to the case where $\dir{C}$ is not necessarily all of $\fg^*$:

\begin{definition}\label{def:unimodular_cone}
  A polyhedral cone $C\subseteq \fg^*$ is called 
  \begin{itemize}
      \item \textit{simple} if it is the conic hull of a basis of $\dir{C}$ and 
      \item \textit{unimodular} if it is rational and the conic hull of a $\ZZ$-basis of $\fgz^*\cap \dir{C}$.
  \end{itemize}
\end{definition}
\begin{definition}\label{def:unimodular_polytope}
    A polytope $\Delta\subseteq \fg^*$ is called 
    \begin{itemize}
        \item \textit{simple} if the local cone $C_\xi \Delta$ is simple for all vertices $\xi\in \text{Vert}(\Delta)$ and
        \item \textit{unimodular}
if it is rational and the local cone $C_\xi\Delta$ is unimodular for all vertices $\xi\in \text{Vert}(\Delta).$
    \end{itemize}
\end{definition}

\begin{corollary}
\label{cor:facerecursive}
    Let $P \subseteq \fg^*$ be a polyhedral cone or a polytope and let $F\subseteq P$ be a face. 
    \begin{enumerate}
        \item If $P$ is simple, then $F$ is simple.
        \item If $P$ is unimodular, then $F$ is unimodular.
    \end{enumerate}
\end{corollary}
\begin{proof}
    By \cref{lem:coneface_to_polytopeface} it is enough to show the result for a simple polyhedral cone $P=C$. Let $\lambda_1,...,\lambda_n\in \fg^*$ be a basis of $\dir{C}$ generating it. Then the conic hull of a subset of $\{\lambda_1,...,\lambda_n\}$ is a face of $C$. Conversely, any face $F$ of $C$ is the conic hull of the generators it contains.
    It follows that the faces of simple polyhedral cones are simple polyhedral cones. Similarly we see that the faces of unimodular polyhedral cones are unimodular polyhedral cones.
\end{proof}


\subsection{Intersections of Polyhedra} \label{sec:polyehedronintersection}
\begin{lemma}[Intersections of Polyhedra are Clean] \label{lem:cleanintersection}
    Let $P_1,P_2\subseteq \fg^*$ be polyhedra. Then, $P_1\cap P_2$ is a polyhedron and for each $\xi \in P_1\cap P_2$, we have
    \begin{equation*}
        L_\xi(P_1\cap P_2) = L_\xi P_1 \cap L_\xi P_2 
    \end{equation*}
    and
    \begin{equation*}
        \qquad C_\xi (P_1\cap P_2) = C_\xi P_1 \cap C_\xi P_2.
    \end{equation*}
\end{lemma}
\begin{proof}
    This follows immediately from the explicit expressions in \cref{lem:explicitexpressions}.
\end{proof}

\begin{definition} \label{def:relativetransversality}
    Let $P_1$ and $P_2$ be two polyhedra in $\fg^*$. We say that they 
    \begin{itemize}
        \item intersect \emph{relatively transversely} at $\xi \in P_1 \cap P_2$ when
    \begin{equation*}
        L_\xi P_1 + L_\xi P_2 = \dir{P_1} + \dir{P_2}.
    \end{equation*}
    \item intersect \emph{integrally transversely} at $\xi \in P_1 \cap P_2$ if both $P_1$ and $P_2$ are rational and
    \begin{equation*}
        (L_\xi P_1\cap \fgz^*) + (L_\xi P_2 \cap \fgz^*) = (\dir{P_1}\cap\fgz^*) + (\dir{P_2}\cap \fgz^*).
    \end{equation*}
    \end{itemize}
    The polyhedra intersect \emph{relatively/integrally transversely}, denoted $P_1\pitchfork P_2$ and $P_1\pitchfork_{_\ZZ} P_2$ respectively, when they intersect relatively/integrally transversely at all their intersection points.
\end{definition}

Integral transversality as defined above implies relative transversality, since the vector spaces are rational and hence are spanned by the corresponding lattices. In particular, the following results also hold for rational polyhedra which intersect integrally transversely:
\begin{lemma}[Consequences of Relative Transversality] \label{lem:consequencesoftransversality}
    Let $P_1,P_2\subseteq \fg^*$ be polyhedra such that $P_1\cap P_2\neq \emptyset$ and $P_1\pitchfork P_2$. Then
    \begin{enumerate}
        \item $\operatorname{relint}(P_1) \cap \operatorname{relint}(P_2)\neq \emptyset$ and
        \item $\dir{P_1\cap P_2} =\dir{P_1} \cap \dir{P_2}$
        \end{enumerate}
        Now set $d= \dim(\dir{P_1} + \dir{P_2})$. Then
        \begin{enumerate}
            \setcounter{enumi}{2}
        \item at any $\xi \in P_1\cap P_2$ we have $\dim(L_\xi(P_1 \cap P_2))=\dim(L_\xi P_1)+\dim(L_\xi P_2)-d.$ 
        \item the vertices of $P_1\cap P_2$ are exactly the points $\xi\in P_1\cap P_2$ where $\dim(L_\xi P_1)+\dim(L_\xi P_2)=d.$
    \end{enumerate}
\end{lemma}
\begin{proof}
    \begin{enumerate}
        \item Assume by contradition that $\operatorname{relint}(P_1) \cap \operatorname{relint}(P_2)= \emptyset$. By Rockafellars separation criterion \cite[Theorem 11.3]{rockafellar1970convex}, there exists a hyperplane in $\dir{P_1}+\dir{P_2}$ which properly separates $P_1$ and $P_2$. In other words, there exist $v\in \fg$ and $c\in \RR$ such that
        \begin{equation*}
            \expval{\xi,v}\geq c, \;\forall \xi \in P_1 \qquad \text{and} \qquad \expval{\zeta,v} \leq c, \; \forall \zeta\in P_2
        \end{equation*}
        and $\expval{\cdot ,v}:\fg^* \to \RR$ is non-constant on $P_1 \cup P_2.$
        
        Let $\eta \in P_1\cap P_2$ and note that by the above, we must have $\expval{\eta,v}=c$. But for any $\varphi\in L_\eta P_1$, we have $\eta \pm t\varphi\in P_1$ for $\abs{t}$ sufficiently small. The separation inequality then reads
        \begin{equation*}
            c \leq \expval{\eta \pm t\varphi,v} = \expval{\eta,v}\pm t\expval{\varphi,v} = c\pm t\expval{\varphi,v}
        \end{equation*}
        which is only possible if $\expval{\varphi,v}=0$. We deduce that $L_\eta P_1\subseteq \ker(\expval{\cdot ,v})$. Analoguously, we show that $L_\eta P_2 \subset \ker(\expval{\cdot , v})$ and conclude that
        \begin{equation*}
            L_\eta P_1 + L_\eta P_2 \subset \ker(\expval{\cdot ,v}).
        \end{equation*}
        By the transversality condition we must thus have 
        \begin{equation*}
            \dir{P_1}+\dir{P_2} \subset \ker(\expval{\cdot ,v})
        \end{equation*}
        which contradicts $\expval{\cdot ,v}$ being non-constant on $P_1 \cup P_2.$
        \item This follows from applying \cref{lem:cleanintersection} at a $\xi \in \operatorname{relint}(P_1) \cap \operatorname{relint}(P_2),$ which exists by 1. 
        \item We compute
        \begin{align*}
            \dim(L_\xi (P_1 \cap P_2))&= \dim(L_\xi P_1 \cap L_\xi P_2) \\
            &= \dim(L_\xi P_1) + \dim(L_\xi P_2) - \dim(\underbrace{L_\xi P_1 + L_\xi P_2}_{=\dir{P_1}+\dir{P_2}}) \\
            &= \dim(L_\xi P_1) + \dim(L_\xi P_2) - d.
        \end{align*}
        \item $\eta\in P_1\cap P_2$ is a vertex if and only if $L_\eta(P_1\cap P_2) = \{0\}.$ The claim is then immediate by 3. 
        \qedhere
    \end{enumerate}
\end{proof}

\subsubsection*{Checking the Transversality Conditions}
We will now show how to check whether two polyhedra intersect transversely in practice. The starting point is the following result, which allows to reduce to the condition at a vertex:

\begin{lemma}[Worst Case Lemma] \label{lem:worstcase_unified}
    Let $P_1,P_2\subseteq \fg^*$ be polyhedra such that $P_1\cap P_2\neq \emptyset$ has at least one vertex. For every $\xi \in P_1\cap P_2$, there exists a vertex $\eta\in \textup{Vert}(P_1\cap P_2)$ of the intersection such that
    \begin{equation*}
        L_\eta P_1\subseteq L_\xi P_1 \qquad \text{and} \qquad L_\eta P_2 \subseteq L_\xi P_2.
    \end{equation*}
\end{lemma}
\begin{proof}
    For $i=1,2$, let $F_i$ be the unique face of $P_i$ containing $\xi\in P_1\cap P_2$ in its interior.
    The intersection $F_1\cap F_2$ is a non-empty face of the polyhedron $P_1 \cap P_2.$ Let $\eta$ be any vertex of $F_1\cap F_2$.\footnote{A polyhedron has a vertex if and only if it does not contain a line (see e.g. \cite[Theorem 4.8]{barvinok2008integer}). If $F_1\cap F_2$ does not contain a vertex, it contains a line. Since $F_1\cap F_2\subseteq P_1\cap P_2$ this implies that $P_1\cap P_2$ contains a line, contradicting the assumption that $P_1\cap P_2$ has a vertex.} Then:
    \begin{itemize}
        \item Since $\eta\in F_i$ and $\xi \in \textup{int}(F_i)$, we have $L_\eta P_i \subseteq L_\xi P_i.$
        \item Since $\eta$ is a vertex of the face $F_1\cap F_2$, it is also a vertex of $P_1\cap P_2.$ \qedhere
    \end{itemize}
\end{proof}
\begin{lemma}[Criterion for Relative Transversality] \label{lem:criteriafortransversality}
    Let $P_1,P_2\subset \fg^*$ be polyhedra such that $P_1\cap P_2\neq \emptyset$ has at least one vertex and write
    \begin{equation*}
        d:= \dim(\dir{P_1}+\dir{P_2}).
    \end{equation*}
    The following are equivalent:
    \begin{enumerate}
        \item $P_1$ and $P_2$ intersect relatively transversely;
        \item If two faces $F_1\subseteq P_1$ and $F_2\subseteq P_2$ intersect, their dimensions sum at least to $d:$
        \begin{equation*}
            \dim(L_\xi P_1) +\dim(L_\xi P_2) \geq d \qquad \forall \xi \in P_1 \cap P_2.
        \end{equation*}
    \end{enumerate}
\end{lemma}
\begin{proof}
    \begin{description}
        \item[1. $\implies$ 2.] Let $\eta\in \textup{Vert}(P_1\cap P_2)$ be any vertex. The transversality condition at $\eta$ implies that \begin{equation*}
            \dim(L_\eta P_1) + \dim(L_\eta P_2) = d + \dim(L_\eta (P_1 \cap P_2)) = d
        \end{equation*}
        where the last equality follows from the fact that $\eta$ is a vertex of $P_1\cap P_2.$

        Now let $\xi \in P_1\cap P_2$ be arbitrary. By \cref{lem:worstcase_unified} there is a vertex $\eta\in \textup{Vert}(P_1\cap P_2)$ such that $L_\eta P_i \subseteq L_\xi P_1$. In particular, 
        \begin{equation*}
            \dim(L_\xi P_1) + \dim(L_\xi P_2) \geq \dim(L_\eta P_1) + \dim(L_\eta P_2) = d.
        \end{equation*}
        \item[2. $\implies$ 1.] For any $\xi \in P_1 \cap P_2$, \cref{lem:worstcase_unified} gives $\eta\in \textup{Vert}(P_1\cap P_2)$ with $L_\eta P_i \subseteq L_\xi P_i.$ Hence
        \begin{align*}
            \dim(L_\xi P_1 + L_\xi P_2) &\geq \dim(L_\eta P_1 + L_\eta P_2) \\
            &=\dim(L_\eta P_1) +\dim(L_\eta P_2) \\
            &\geq d \\
            &= \dim(\dir{P_1}+\dir{P_2}).
        \end{align*}
        Hence, the obvious inclusion $L_\xi P_1 + L_\xi P_2 \subseteq \dir{P_1} + \dir{P_2}$ is an equality. \qedhere
    \end{description}
\end{proof}
\begin{lemma}[Criterion for Integral Transversality] \label{lem:criteriaforintegraltransversality}
    Let $P_1,P_2\subset \fg^*$ be rational polyhedra such that $P_1\cap P_2 \neq \emptyset$ has at least one vertex. The following are equivalent:
    \begin{enumerate}
        \item $P_1$ and $P_2$ intersect integrally transversely;
        \item $P_1$ and $P_2$ intersect integrally transversely at all vertices of $P_1\cap P_2.$
    \end{enumerate}
\end{lemma}
\begin{proof}
    Let $\xi \in P_1 \cap P_2$ be an arbitrary point. By the Worst Case \cref{lem:worstcase_unified}, since $P_1 \cap P_2$ has at least one vertex, there exists $\eta \in \operatorname{Vert}(P_1 \cap P_2)$ such that: 
    $$L_\eta P_1 \subseteq L_\xi P_1 \quad \text{and} \quad L_\eta P_2 \subseteq L_\xi P_2.$$
    Intersecting these inclusions with the lattice $\mathfrak{g}^*_{\mathbb{Z}}$ preserves inclusion:
    $$L_\eta P_1 \cap \mathfrak{g}^*_{\mathbb{Z}} \subseteq L_\xi P_1 \cap \mathfrak{g}^*_{\mathbb{Z}} \quad \text{and} \quad L_\eta P_2 \cap \mathfrak{g}^*_{\mathbb{Z}} \subseteq L_\xi P_2 \cap \mathfrak{g}^*_{\mathbb{Z}}.$$
    Summing these two subgroup inclusions gives:
    $$(L_\eta P_1 \cap \mathfrak{g}^*_{\mathbb{Z}}) + (L_\eta P_2 \cap \mathfrak{g}^*_{\mathbb{Z}}) \subseteq (L_\xi P_1 \cap \mathfrak{g}^*_{\mathbb{Z}}) + (L_\xi P_2 \cap \mathfrak{g}^*_{\mathbb{Z}}).$$
    Conversely, because $L_\xi P_i \subseteq \dir{P_i}$ for $i = 1, 2$, taking lattice intersections yields:
    $$(L_\xi P_1 \cap \mathfrak{g}^*_{\mathbb{Z}}) + (L_\xi P_2 \cap \mathfrak{g}^*_{\mathbb{Z}}) \subseteq (\dir{P_1} \cap \mathfrak{g}^*_{\mathbb{Z}}) + (\dir{P_2} \cap \mathfrak{g}^*_{\mathbb{Z}}).$$
    Chaining these inclusions with the assumption at $\eta$ yields:
    \begin{align*}
        (\dir{P_1} \cap \mathfrak{g}^*_{\mathbb{Z}}) + (\dir{P_2} \cap \mathfrak{g}^*_{\mathbb{Z}}) &= (L_\eta P_1 \cap \mathfrak{g}^*_{\mathbb{Z}}) + (L_\eta P_2 \cap \mathfrak{g}^*_{\mathbb{Z}}) \\
        &\subseteq (L_\xi P_1 \cap \mathfrak{g}^*_{\mathbb{Z}}) + (L_\xi P_2 \cap \mathfrak{g}^*_{\mathbb{Z}}) \\
        &\subseteq (\dir{P_1} \cap \mathfrak{g}^*_{\mathbb{Z}}) + (\dir{P_2} \cap \mathfrak{g}^*_{\mathbb{Z}}).
    \end{align*}
    Since the outer boundaries of this inclusion chain coincide, every inclusion in the chain must be an equality.
\end{proof}

We deduce that there is a clear algorithm on how to check whether two rational polyhedra $P_1,P_2\subseteq \fg^ *$ intersect integrally transversely:
\begin{enumerate}
    \item Determine $d=\dim(\dir{P_1}+\dir{P_2})$.
    \item Verify that for all faces $F_1\subseteq P_1$ and $F_2\subseteq P_2$ such that $F_1\cap F_2\neq 0$, the dimensions sum up to at least $d$. Deduce by \cref{lem:criteriafortransversality} that $P_1\pitchfork P_2$.
    \item Determine the vertices of $P_1\cap P_2$. By \cref{lem:consequencesoftransversality}, these are precisely the points where the face-dimensions sum exactly to $d$.
    \item Verify the integral transversality conditions at all vertices of $P_1\cap P_2.$ Deduce by \cref{lem:criteriaforintegraltransversality} that $P_1 \pitchfork_{_\ZZ}P_2$.
\end{enumerate}

\subsubsection*{Transverse Slices of Polytopes}
Let us now consider the special case where one of the polyhedra is an affine subspace $A\subseteq \fg^ *$ and the other is a polytope $\Delta\subset \fg^ *$. If $A\pitchfork \Delta$, we call $A\cap \Delta$ a \emph{transverse slice} of the polytope and if $A\pitchfork_{_\ZZ}\Delta$, we call $A\cap \Delta$ an \emph{integrally transverse slice} of $\Delta.$ For the convenience of the reader, we state \cref{lem:consequencesoftransversality} and \cref{lem:criteriafortransversality} for this special case:

\begin{corollary} \label{cor:consequencesofaffinetransversality}
    Let $\Delta\subset \fg^*$ be a polytope and
    let $A\subset \fg^*$ be a translate of a subspace $B$
    such that $A\pitchfork \Delta$ and $A\cap \Delta\neq \emptyset$.
    Then 
    \begin{enumerate}
         \item $A\cap \operatorname{relint}(\Delta) \neq \emptyset$;
        \item $\dir{A\cap\Delta} = B\cap\dir{\Delta}$.
    \end{enumerate}
    Now set $k=\dim(B+\dir{\Delta})-\dim(B)$. Then
    \begin{enumerate}
    \setcounter{enumi}{2}
    \item
    at any $\xi \in A \cap \Delta$ we have $\dim(L_\xi(A\cap \Delta))=\dim(L_\xi\Delta)-k$;
    \item
    the vertices of $A \pitchfork \Delta$ are exactly
the points where $A$ intersects the $k$-dimensional faces of $\Delta$.
    \end{enumerate}
\end{corollary}
\begin{corollary}\label{cor:criteriaforaffinetransversality}
    Let $\Delta\subseteq \fg^*$ be a polytope, $A\subseteq \fg^*$ an affine subspace with direction $B$ and $k=\dim(B+\dir{\Delta})-\dim(B)$. The following are equivalent:
    \begin{enumerate}
        \item $A\pitchfork \Delta$;
        \item Any face of $\Delta$ that intersects $A$ has dimension at least $k$:
        \begin{equation*}
            \dim(L_\xi \Delta) \geq k \qquad \forall \xi \in A\cap \Delta.
        \end{equation*}
    \end{enumerate}
\end{corollary}
Both statements follow immediately from the corresponding statements for polyhedra and the observations that $\operatorname{relint}(A) = A$, $\dir{A}=B$ and $L_\xi A = B$ for all $\xi \in A.$ The crucial observation is now that taking transvers/integrally transverse slices preserves simplicity/unimodularity:
\begin{proposition}[Transverse Slices] \label{prop:transverseslices}
    Let $\Delta\subseteq \fg^*$ be a polytope and $A\subset \fg^*$ an affine subspace intersecting $\Delta$.
    \begin{enumerate}
        \item If $\Delta$ is simple and $A\pitchfork \Delta$, the intersection $A\cap \Delta$ is simple. 
        \item If $\Delta$ is unimodular and $A\pitchfork_{_\ZZ}\Delta$, the intersection $A\cap \Delta$ is unimodular. 
    \end{enumerate}
\end{proposition}

\begin{proof}
Since the proofs of 1. and 2. are essentially identical, we only give a proof for 2.

Let $B$ be the direction subspace of $A$, and let
$\Lambda:=\fgz^*\cap\dir{\Delta}$.
By \cref{def:unimodular_polytope}, we need to show that,
for any vertex $\xi$ of $A\cap\Delta$, the local cone
$C_\xi(A\cap\Delta)$ is unimodular.

\medskip

\noindent\textit{Step 1: Local data at $\xi$.}

\noindent Let $F$ be the unique face of $\Delta$ with $\xi\in\text{relint}(F)$,
set $m:=\dim F$ and $n:=\dim\Delta$, and fix any vertex
$\eta\in\text{Vert}(F)\subseteq\text{Vert}(\Delta)$.
Since $\Delta$ is unimodular, $C_\eta\Delta$ is generated by a
$\ZZ$-basis $\lambda_1,\ldots,\lambda_n$ of $\Lambda$.
By \cref{lem:coneface_to_polytopeface}, $C_\eta F$ is a face of
$C_\eta\Delta$, so after reindexing we may assume
\[
C_\eta F=\co{\lambda_1,\ldots,\lambda_m}.
\]
By \cref{cor:facerecursive}, $C_\eta F$ is itself unimodular, so
$\lambda_1,\ldots,\lambda_m$ is a $\ZZ$-basis of
$\Lambda\cap\dir{\lambda_1,\ldots,\lambda_m}$.
Since
\[
L_\xi\Delta = \spn{\RR}{F-\xi}= \spn{\RR}{F-\eta}
= \spn{\RR}{\lambda_1,\ldots,\lambda_m},
\]
we have
\begin{equation*}
\Lambda_{_F} :=\Lambda\cap L_\xi\Delta =\ZZ\lambda_1+\cdots+\ZZ\lambda_m .
\end{equation*}
Since $\xi$ and $\eta$ lie in the same face $F$,
the $n-m$ facets of $\Delta$ through $\xi$ are exactly
the facets through $\eta$ that contain $F$,
namely those spanned by
$\{\lambda_i\}_{i\neq j}$ for $j=m+1,\ldots,n$.
Since $\lambda_1,\ldots,\lambda_n$ is a basis
of $\dir{\Delta}$, writing $x=\sum_{i=1}^n c_i\lambda_i$,
the simplicity of $C_\eta\Delta$ gives
\[
x\in C_\xi\Delta \; \iff \; c_i\geq 0 \text{ for } i=m+1,\ldots,n,
\]
with $c_1,\ldots,c_m$ unrestricted; that is,
\begin{equation}
\label{eq:localcone}
C_\xi\Delta = L_\xi\Delta + \co{\lambda_{m+1},\ldots,\lambda_n}.
\end{equation}

\medskip

\noindent\textit{Step 2: Reduction to a lattice statement.}

\noindent From \cref{lem:cleanintersection}, we get
\begin{equation}
\label{eq:conecap}
C_\xi(A\cap\Delta) = B\cap C_\xi\Delta \qq{and} L_\xi(A\cap \Delta) = B\cap L_\xi \Delta = \{0\},
\end{equation}
where the last equality is due to $\xi$ being a vertex of $A\cap\Delta$. Moreover, \cref{cor:consequencesofaffinetransversality} gives $\dim(B\cap\dir{\Delta}) = \dim\Delta - \dim L_\xi\Delta=n-m$.

Integral transversality of $A$ and $\Delta$ at $\xi$
(\cref{def:int_transv}) reads
\[
(B\cap\fgz^*)+(L_\xi\Delta\cap\fgz^*) = (B\cap\fgz^*)+\Lambda.
\]
Intersecting both sides with $\Lambda$ and using the modular law
$\Lambda\cap(X+Y)=(\Lambda\cap X)+Y$ (valid since
$L_\xi\Delta\cap\fgz^*\subseteq\Lambda$) -- as in the proof
of the Recursive Aspects of Reduction corollary above --
this becomes
\begin{equation}
\label{eq:lambdaB}
\Lambda_{_B}+\Lambda_{_F}=\Lambda, \qquad \text{where } \Lambda_{_B}:=B\cap\Lambda.
\end{equation}

\medskip

\noindent\textit{Step 3: $\Lambda_{_B}$ is a lattice complement to
$\Lambda_{_F}$.}

\noindent We claim that the homomorphism
$\Lambda_{_B}\to\Lambda/\Lambda_{_F}$ is an isomorphism.
It is surjective by \eqref{eq:lambdaB}.
It is injective:
if $\beta\in\Lambda_{_B}$ maps to $0$,
then $\beta \in \Lambda_{_F} \subseteq L_\xi\Delta$,
while also $\beta \in B$, 
so $\beta \in B\cap L_\xi\Delta = \{0\}$.

Since $\lambda_{m+1},\ldots,\lambda_n$ project to a $\ZZ$-basis of
the free abelian group $\Lambda/\Lambda_{_F}$,
there are unique elements $\beta_{m+1},\ldots,\beta_n\in\Lambda_{_B}$ with
\begin{equation}
\label{eq:betadef}
\beta_i-\lambda_i\in\Lambda_{_F} \qquad (i=m+1,\ldots,n),
\end{equation}
and these form a $\ZZ$-basis of $\Lambda_{_B}$.

\medskip

\noindent\textit{Step 4: Identifying the local cone.}

\medskip

We claim that
\begin{equation}
\label{eq:claim_cap}
B\cap C_\xi\Delta = \co{\beta_{m+1},\ldots,\beta_n}.
\end{equation}
Indeed, by \eqref{eq:localcone},
\[
\dir{\Delta} =
L_\xi\Delta\oplus\spn{\RR}{\lambda_{m+1},\ldots,\lambda_n},
\]
and by Step 2 the subspace
\[
B':=B\cap\dir{\Delta}
\]
satisfies
$B'\cap L_\xi\Delta=\{0\}$ and $\dim B'=n-m$, so the projection
\[
\pi : \dir{\Delta} \longrightarrow \spn{\RR}{\lambda_{m+1},\ldots,\lambda_n}
\]
along $L_\xi\Delta$ restricts to a linear isomorphism
\[
\pi|_{B'}:B'\xrightarrow{\ \sim\ }
\spn{\RR}{\lambda_{m+1},\ldots,\lambda_n}.
\]
Write $x=\ell+c\in C_\xi\Delta$ with $\ell\in L_\xi\Delta$
and $c\in \co{\lambda_{m+1},\ldots,\lambda_n}$.  Then $\pi(x)=c$.
Therefore,
\begin{equation}
\label{eq:B'}
x\in B' \; \iff \; x=(\pi|_{B'})^{-1}(c)
\end{equation}
and
\begin{align*}
& B\cap C_\xi\Delta \\
= \, & B'\cap C_\xi\Delta
& \text{since $C_\xi\Delta \subseteq \dir{\Delta}$} \\
= \, & (\pi|_{B'})^{-1}\bigl(\co{\lambda_{m+1},\ldots,\lambda_n}\bigr)
& \text{by \eqref{eq:B'}} \\
= \, & \co{(\pi|_{B'})^{-1}(\lambda_{m+1}),\ldots,(\pi|_{B'})^{-1}(\lambda_n)}
& \text{by linearity} \\
= \, & \co{\beta_{m+1},\ldots,\beta_n}
& \text{as $\beta_i=(\pi|_{B'})^{-1}(\lambda_i)$ by \eqref{eq:betadef},}
\end{align*}
(for $i=m+1,\ldots,n$), proving \eqref{eq:claim_cap}.

\medskip

\noindent\textit{Conclusion.}

\noindent Combining
\eqref{eq:conecap} with \eqref{eq:claim_cap},
\[
C_\xi(A\cap\Delta) = \co{\beta_{m+1},\ldots,\beta_n}.
\]
By \cref{cor:consequencesofaffinetransversality},
$\dir{A\cap\Delta}=B\cap\dir{\Delta}$, so
\[
\fgz^*\cap\dir{A\cap\Delta}=\fgz^*\cap B\cap\dir{\Delta}
=\Lambda_{_B}.
\]
By Step 3, $\beta_{m+1},\ldots,\beta_n$ is a $\ZZ$-basis of
$\Lambda_{_B}$.
Therefore, $C_\xi(A\cap\Delta)$ is a unimodular cone.
As $\xi\in\text{Vert}(A\cap\Delta)$ was arbitrary,
$A\cap\Delta$ is a unimodular polytope. \qedhere
\end{proof}

\begin{lemma}[Recursive Aspects of Integral Transversality]
    Let $\Delta \subseteq \fg^*$ be a unimodular polytope,
    $A' \subseteq A \subseteq \fg^*$ rational affine subspaces,
    and $A \pitchfork_{_\ZZ} \Delta$. 
    Then
    \[
    A' \pitchfork_{_\ZZ} \Delta
    \; \iff \;
    A'\pitchfork_{_\ZZ} (A \cap \Delta).
    \]
\end{lemma}

\begin{proof}
Since $A'\subseteq A$, throughout we use $A' \cap (A \cap \Delta) = A' \cap \Delta$.
Let $B' \subseteq B\subseteq \fg^*$ be the direction subspaces of $A',A$.

\begin{description}
     \item[$\Longrightarrow$]
     Fix $\xi \in A'\cap\Delta$.
Since $A'\pitchfork_\ZZ \Delta$, \cref{def:int_transv} at $\xi$ gives
\[
(B'\cap\fgz^*)+(L_\xi\Delta\cap\fgz^*)
= (B'\cap\fgz^*)+(\dir{\Delta}\cap\fgz^*).
\]
Intersecting both sides of the above with $B\cap\fgz^*$
and applying modular distributivity
(since $B'\cap\fgz^* \subseteq B\cap\fgz^*$),
we obtain
\[
(B'\cap\fgz^*)+(B\cap L_\xi\Delta\cap\fgz^*)
= (B'\cap\fgz^*)+(B\cap\dir{\Delta}\cap\fgz^*).  \tag{$*$}
\]
Therefore,
\begin{align*}
& \; (B'\cap\fgz^*)+(L_\xi(A\cap\Delta)\cap\fgz^*) \\
= & \; (B'\cap\fgz^*)+(B\cap L_\xi\Delta\cap\fgz^*)
& \text{by \cref{lem:cleanintersection}} \\
= & \; (B'\cap\fgz^*)+(B\cap\dir{\Delta}\cap\fgz^*)
& \text{by $(*)$} \\
= & \; (B'\cap\fgz^*)+(\dir{A\cap\Delta}\cap\fgz^*)
& \text{by \cref{cor:consequencesofaffinetransversality},}
\end{align*}
where we may assume $A \cap \Delta \neq \emptyset$.
This yields the integral transversality for
$A'\pitchfork_\ZZ(A\cap\Delta)$ at $\xi$,
where $\xi \in A' \cap \Delta$ was arbitrary.

\item[$\Longleftarrow$]
Fix $\xi\in A'\cap\Delta$.
Since $A\pitchfork_\ZZ\Delta$, \cref{def:int_transv} at $\xi$ gives
\[
(B\cap\fgz^*) + (L_\xi\Delta\cap\fgz^*)
=(B\cap\fgz^*) + (\dir{\Delta}\cap\fgz^*).
\]
Intersecting the above with $\dir{\Delta}\cap\fgz^*$
and applying modular distributivity
(since $L_\xi\Delta\cap\fgz^* \subseteq \dir{\Delta}\cap\fgz^*$),
we obtain
\[
(B\cap\dir{\Delta}\cap\fgz^*) + (L_\xi\Delta\cap\fgz^*)
=(\dir{\Delta}\cap\fgz^*).  \tag{$\dagger$}
\]
Therefore,
\begin{align*}
& \; (B'\cap\fgz^*) + (\dir{\Delta}\cap\fgz^*) \\
= & \; (B'\cap\fgz^*) + (B \cap \dir{\Delta}\cap\fgz^*)
+ (L_\xi\Delta \cap \fgz^*)
& \text{by $(\dagger)$} \\
= & \; (B'\cap\fgz^*) + (\dir{A \cap \Delta}\cap\fgz^*)
+ (L_\xi\Delta \cap \fgz^*)
& \text{by \cref{cor:consequencesofaffinetransversality}} \\
= & \; (B'\cap\fgz^*) + (L_\xi(A \cap \Delta)\cap\fgz^*)
+ (L_\xi\Delta \cap \fgz^*)
& \text{by $A'\pitchfork_\ZZ(A\cap\Delta)$ at $\xi$} \\
= & \; (B'\cap\fgz^*) + (L_\xi\Delta \cap \fgz^*)
& \text{as $A \cap \Delta \subseteq \Delta$.}
\end{align*}
This yields the integral transversality condition for
$A'\pitchfork_\ZZ\Delta$ at $\xi$,
where $\xi\in A'\cap\Delta$ was arbitrary.
\qedhere
\end{description}
\end{proof}

\subsubsection*{Reducing Intersections to Slices}
The case of relative intersections of polyhedra can be brought back to the case of transverse slices. The reason is the following analogue of an equivalent characterisation of standard transversality.
\begin{lemma}[Diagonal Characterisation of Relative Transversality]
\label{lem:diagonaltransversality}
Let $P_1,P_2\subseteq \fg^*$ be polyhedra. Under the identification
\begin{align*}
    P_1\cap P_2 &\to 
(P_1\times P_2)\cap\text{Diag}_{\fg^*},\\
\xi&\mapsto(\xi,\xi),
\end{align*}
the following are equivalent:
\begin{enumerate}
    \item $P_1$ and $P_2$ intersect relatively transversely;
    \item $P_1\times P_2$ and $\text{Diag}_{\fg^*}$ intersect relatively
    transversely in $\fg^*\times \fg^*$.
\end{enumerate}
\end{lemma}

\begin{proof}
Fix $\xi\in P_1\cap P_2$ and consider
\begin{align*}
    q:\fg^*\times \fg^* &\longrightarrow \fg^*,\\
(x,y)&\longmapsto x-y.
\end{align*}
Its kernel is $\text{Diag}_{\fg^*}$. Moreover,
\[
L_{(\xi,\xi)}(P_1\times P_2)
    =L_\xi P_1\times L_\xi P_2
\]
and
\[
\dir{P_1\times P_2}
    =\dir{P_1}\times\dir{P_2}.
\]
The relative transversality condition of $P_1\times P_2$
and $\text{Diag}_{\fg^*}$ at $(\xi,\xi)$ is thus
\[
(L_\xi P_1\times L_\xi P_2)+\text{Diag}_{\fg^*}
=
(\dir{P_1}\times\dir{P_2})+\text{Diag}_{\fg^*}.
\]
Both sides contain $\ker(q)$, so they are equal if and only if
their images under $q$ are equal. These images are respectively
\[
L_\xi P_1+L_\xi P_2
\qquad\text{and}\qquad
\dir{P_1}+\dir{P_2}. \qedhere
\]
\end{proof}
The same argument can be carried out on the level of lattices to show that integrally transverse intersections can be viewed as integrally transverse slices. Combining this with \cref{prop:transverseslices}, we deduce that

\begin{corollary} \label{cor:polytopeintersection}
    Let $\Delta_1,\Delta_2\subseteq \fg^*$ be intersecting polytopes. 
    \begin{enumerate}
        \item If both $\Delta_1$ and $\Delta_2$ are simple and $\Delta_1 \pitchfork \Delta_2$, the intersection $\Delta_1\cap \Delta_2$ is simple.
        \item If both $\Delta_1$ and $\Delta_2$ are unimodular and $\Delta_1 \pitchfork_{_\ZZ} \Delta_2$, the intersection $\Delta_1\cap \Delta_2$ is unimodular.
    \end{enumerate}
\end{corollary}


\section{Pontryagin Duality}
\label{sec:pontryagin}

In this appendix, we record some basic facts about Pontryagin duality
and the annihilator mechanism for locally compact abelian groups. We are not aiming for a self-contained introduction but rather a list of important results with precise references to the relevant literature.

\begin{definition}[{\cite[Definition 1.2.1]{rudin2017fourier}}]
Let $G$ be a locally compact abelian group. The \emph{character group} of $G$ is defined as $$\widehat{G} = \Hom(G,S^1).$$ For a morphism $\varphi:G\to H$, there is an induced morphism $\widehat{\varphi}:\widehat{H} \to \widehat{G}$ given by $\chi \mapsto \chi \circ \varphi$.
\end{definition}
\begin{theorem}[{\cite[Theorem 1.2.6]{rudin2017fourier}}]
    Equipped with the compact-open topology, $\widehat{G}$ is a locally compact abelian group. 
\end{theorem}

\begin{theorem}
[{\cite[Theorem 7.63]{hofmann2013structure}, \cite[Theorem 1.7.2]{rudin2017fourier}}]
For every locally compact abelian group $G$, the evaluation morphism \begin{align*}
    \eta_G:G&\to \widehat{\widehat{G}} \\
    g &\mapsto (\chi \mapsto \chi(g))
\end{align*} is an isomorphism of topological groups.
\end{theorem}

\begin{theorem}[{\cite[1.7.3 (a)]{rudin2017fourier}}]
    \begin{itemize}
    \item $G$ is compact if and only if $\widehat{G}$ is discrete.
    \item $G$ is discrete if and only if $\widehat{G}$ is compact.
\end{itemize}
\end{theorem}

\begin{definition}
    A sequence of topological groups $0\to H \hookrightarrow G \twoheadrightarrow Q \to 0$ is called \emph{strictly exact} if \begin{enumerate}
        \item it is exact as a sequence of abelian groups,
        \item $H\hookrightarrow G$ is an embedding and
        \item $G\twoheadrightarrow Q$ is a quotient morphism. 
    \end{enumerate}
\end{definition}
\begin{remark}
    Any sequence that is exact in the category of compactly generated\footnote{``Compactly generated'' ensures that also the character group is a (compactly generated) abelian Lie group.} abelian Lie groups is strictly exact. 
\end{remark}
\begin{theorem}[{\cite[Theorem 7.64 (vi)]{hofmann2013structure}}] \label{thm:exactness}
Let $0\to H \hookrightarrow G \twoheadrightarrow Q \to 0$ be a strictly exact sequence of locally compact abelian groups. Then the dual sequence $0 \to \widehat{Q} \hookrightarrow \widehat{G} \twoheadrightarrow\widehat{H} \to 0$ is also strictly exact.
\end{theorem}

\begin{definition}[{\cite[Definition 2.1.1]{rudin2017fourier}, \cite[Definition 7.12]{hofmann2013structure}}]
    Let $G$ be a locally compact abelian group and $\widehat{G}$ be its character group. The \emph{annihilator of a closed subgroup $H\subset G$ in $\widehat{G}$} is $$H^0 := \{\chi\in \widehat{G} \mid \chi\vert_H=1\}.$$ Similarly, the \emph{annihilator of a closed subgroup $A\subset \widehat{G}$ in $G$} is $${}^0A := \bigcap_{\chi\in A}\ker(\chi).$$
\end{definition}
\begin{theorem}[Annihilator Mechanism {\cite[Theorem 7.64 (iv),(v) and (vii)]{hofmann2013structure}}]
    \label{thm:annihilatormechanism} Let $G$ be a locally compact abelian group and let $\widehat{G}$ be its character group. 
    \begin{enumerate}
        \item If $H\subset G$ is a closed subgroup, then ${}^0(H^0)= H$. 
        \item The function $H\mapsto H^0$ maps the lattice of closed subgroups of $G$ antiisomorphically (i.e.\ inclusion-reversing) onto the lattice of closed subgroups of $\widehat{G}$. 
        
        \item Let $\{H_j\mid j\in J\}$ be a family of closed subgroups of $G$. Then $$\left(\sum_{j\in J}H_j\right)^0 =  \bigcap_{j\in J}H_j^0\qq{and}\left(\bigcap_{j\in J}H_j\right)^0 =  \overline{\sum_{j\in J}H_j^0}.$$
    \end{enumerate}
\end{theorem}

\begin{proposition}[{\cite[Theorem 7.64 (i),(ii)]{hofmann2013structure}}] \label{prop:interpretation}
    Let $G$ be a locally compact abelian group and let $\widehat{G}$ be its character group. 
    \begin{enumerate}
        \item If $H\subset G$ is a closed subgroup, there are canonical isomorphisms
        \begin{equation*}
            \widehat{G/H} \cong H^0 \qq{and} \widehat{G}/H^0 \cong \widehat{H},
        \end{equation*}
        induced by $H\hookrightarrow G$ and $G\twoheadrightarrow G/H$ respectively. 
        \item For closed subgroups $H_1 \subset H_2 \subset G$, there is a canonical isomorphism
        \begin{equation*}
            H_1^0 / H_2^0 \cong \widehat{H_2/H_1}
        \end{equation*}
    \end{enumerate}
\end{proposition}
\begin{proof} For convenience, we quickly show  how they are deduced from \cref{thm:exactness} and \cref{thm:annihilatormechanism}. 
    \begin{enumerate}
        \item Since $H$ is closed, the sequence
        \begin{equation*}
            0 \to H \overset{i}{\hookrightarrow} G \overset{p}{\twoheadrightarrow} G/H \to 0
        \end{equation*}
        is strictly exact. It follows from \cref{thm:exactness} that the sequence
        \begin{equation*}
            0 \to \widehat{G/H} \overset{\widehat{p}}{\hookrightarrow} \widehat{G} \overset{\widehat{i}}{\twoheadrightarrow} \widehat{ H} \to 0
        \end{equation*}
        is strictly exact. We deduce that
        \begin{equation*}
            \widehat{G/H} \cong \text{Im}(\widehat{p}) = \ker(\widehat{i}) = H^0
        \end{equation*}
        and 
        \begin{equation*}
            \widehat{H} \cong \widehat{G}/(\widehat{G/H}) \cong \widehat{G}/H^0.
        \end{equation*}
        \item The third isomorphism theorem ($(G/H_1)/(H_2/H_1)\cong G/H_2$) is equivalent to exactness of the sequence
        \begin{equation*}
            0\to H_2/H_1 \hookrightarrow G/H_1 \twoheadrightarrow G/H_2 \to 0.
        \end{equation*}
         Dualizing and using $\widehat{G/H_i}\cong H_i^0$ for $i=1,2$, we get the strictly exact sequence
        \begin{equation*}
            0\to H_2^0 \hookrightarrow H_1^0 \twoheadrightarrow \widehat{H_2/H_1} \to 0,
        \end{equation*}
        from where the result is immediate. \qedhere
    \end{enumerate}
\end{proof}

\bibliographystyle{plain}
\bibliography{references.bib}

\end{document}